\documentclass{amsart}
\usepackage{graphicx}
 \usepackage{stmaryrd}
\usepackage{color}
\def\({\left (}
\def\){\right )}
\def\<{\left\langle}
\def\>{\right\rangle}

 \newtheorem{thm}{Theorem}[section]

\newtheorem{lem}[thm]{Lemma}

\newtheorem{rem}[thm]{Remark}

\newcommand{\norm}[1]{\left\Vert#1\right\Vert}
\newcommand{\abs}[1]{\left\vert#1\right\vert}
\newcommand{\set}[1]{\left\{#1\right\}}
\newcommand{\Real}{\mathbb R}

\newcommand{\pfrac}[2]{\frac{\partial #1}{\partial #2}}

\begin{document}
\title{Neck analysis for biharmonic maps}

\numberwithin{equation}{section}
\author{Lei Liu and Hao Yin}

\address{School of Mathematical Sciences,
University of Science and Technology of China, Hefei, China}
\email{LLEI1988@mail.ustc.edu.cn}
\email{haoyin@ustc.edu.cn }

\begin{abstract}
	In this paper, we study the blow up of a sequence of (both extrinsic and intrinsic) biharmonic maps in dimension four with bounded energy and show that there is no neck in this process. Moreover, we apply the method to provide new proofs to the removable singularity theorem and energy identity theorem of biharmonic maps.
\end{abstract}


 \maketitle

 \pagestyle{myheadings} \markright {no neck}

\section{Introduction}

In this paper, we study the neck analysis in the blow-up of a sequence of biharmonic maps in dimension four.

Suppose  $(N,h)$ is a closed Riemannian manifold which is embedded in $\Real^K$. Consider the following functionals for a map $u$ from $\Omega\subset \Real^4$ to $N$,
\begin{equation*}
	H(u)=\int_\Omega \abs{\triangle u}^2 dx,
\end{equation*}
\begin{equation*}
	T(u)=\int_\Omega \abs{\tau(u)}^2 dx
\end{equation*}
and
\begin{equation*}
	E(u)=\int_\Omega \abs{\nabla^u du}^2 dx.
\end{equation*}
Here $\tau(u)$ is the tension field of $u$, or equivalently, the tangential part of $\triangle u$ and $\nabla^u$ is the induced connection of the pullback bundle $u^*TN$. The critical points of all these functionals are called biharmonic maps. Usually, critical points of $H(u)$ are called extrinsic biharmonic maps, because the functional $H(u)$ depends on the particular embedding of $(N,h)$ into the Euclidean space. Critical points of the other two functionals are called intrinsic biharmonic maps. In this paper, we study all three types of biharmonic maps and call the critical points of $T(u)$ intrinsic Laplace biharmonic maps and the critical points of $E(u)$ intrinsic Hessian biharmonic maps.

The study of biharmonic maps was pioneered by Chang, Wang and Yang \cite{CWY}, which is followed by Wang \cite{W04,Wcpam,Wsphere}, Moser \cite{Moser}, Lamm and Rivi\`ere \cite{Lamm}, Struwe \cite{struwe}, Scheven \cite{S1,S} and many others. Most of these work is concerned with the regularity problem of biharmonic maps.

In this paper, we consider a sequence of smooth biharmonic maps $\set{u_i}$ with bounded $W^{2,2}$ norm in the critical dimension. Since the functionals are scaling invariant in this dimension, the theory is similar to the blow-up analysis of harmonic maps in dimension two. Most important of all, an $\varepsilon-$regularity lemma holds for biharmonic maps with small energy (see Theorem \ref{thm:regularity} in Section \ref{sec:pre}). Hence, routine arguments as for harmonic maps in dimension two work for biharmonic maps. It implies that we have a weak limit $u_\infty$ from $\Omega$ to $N$ and finitely many 'bubble' maps, $\omega_i:\Real^4 \to N$. Since none of the biharmonic functionals above is conformally invariant, these bubbles are not biharmonic maps from $S^4$, which is a difference from the theory of harmonic maps.

In the blow-up analysis, we are interested in the following two questions: Is there unaccounted energy in the limit? Is the image of the weak limit and the bubble maps connected? The affirmative answer to the first question is known as energy identity, or energy quantization. This has been proved for critical points of $E(u)$ by Hornung and Moser in \cite{HM}, for critical points of $H(u)$ by Wang and Zheng \cite{WZ11} and \cite{WZ12}. There is also a unified proof for both cases by Laurain and Rivi\`{e}re in \cite{LR}.

The main result of this paper is to give an affirmative answer to the second question which is
first studied in Parker's paper \cite{Parker}  for harmonic maps and in Qing and Tian's paper \cite{QT} 
for approximated harmonic maps and usually known as 'no neck' result. For simplicity, we assume that $\Omega$ is $B^4$, the unit ball in $\Real^4$ and $0\in B^4$ is the only blow-up point. We further assume that there is only one bubble $\omega:\Real^4\to N$. It follows from an induction argument of Ding and Tian \cite{DT} that the result is true for the general case (see also \cite{li} for more details). Precisely, we prove
\begin{thm}
	\label{thm:main}
	Let $u_i$ be a sequence of biharmonic maps from $B^4$ to $N$ satisfying
	\begin{equation}\label{eqn:totalenergy}
		\int_{B_1} \abs{\nabla^2 u_i}^2 +\abs{\nabla u_i}^4 dx <\Lambda
	\end{equation}
	for some $\Lambda>0$.
	Assume that there is a sequence positive $\lambda_i\to 0$ such that
	\begin{equation*}
		u_i(\lambda_i x)\to \omega
	\end{equation*}
	on any compact set $K\subset \Real^4$, that $u_i$ converges weakly in $W^{2,2}$ to $u_\infty$ and that $\omega$ is the only bubble. Then,
	\begin{equation}\label{eqn:noneck}
		\lim_{\delta\to 0} \lim_{R\to \infty} \lim_{i\to \infty} \mbox{osc}_{B_\delta(0)\setminus B_{\lambda_i R}(0)} u_i =0.
	\end{equation}
\end{thm}

Among other things, the theorem implies that $\lim_{\abs{x}\to \infty} \omega(x)$ exists. This observation  enables us to explain the limit in (\ref{eqn:noneck}) as the length of a neck connecting the weak limit and the bubble, which is shown to be zero by the theorem. The proof of this observation is very elementary and is given at the end of Section \ref{sec:pre}.

\begin{rem}
	The biharmonic map problem is not conformally invariant. So we can not use the removable singularity theorem to see that $\lim_{\abs{x}\to\infty} \omega(x)$ exists. Note that this problem is also considered in Lemma 3.4 of \cite{Wsphere}, where the author showed that even if the problem is not conformally invariant, the method of proof can still be used.
\end{rem}

\begin{rem}
	The assumption (\ref{eqn:totalenergy}) is very natural for extrinsic biharmonic maps, since it follows from the bound of $H(u)$ and the fact that $N$ is compact. For intrinsic Hessian biharmonic maps, it is also a reasonable assumption because if the sequence has uniformly bounded energy $E(u_i)$ and uniformly controlled $\int_{\partial B_1} \abs{du_i}^2 d\sigma$, by Lemma 2.2 and Theorem 2.1 of \cite{Moser}, (\ref{eqn:totalenergy}) holds. However, for intrinsic Laplace biharmonic maps, such an assumption is rather strong and unexpected. The reason is well explained in \cite{Moser}.
	We still include this case, simply to show that our proof is robust and can be applied to a variety of equations.
\end{rem}

The proof of Theorem \ref{thm:main} requires refined understanding of the maps $u_i$ in the neck region $B_\delta\setminus B_{\lambda_i R}$. As a byproduct of this understanding, we give new proofs to other known results in the field of neck analysis. The first one is the following energy identity result, which was proved for extrinsic biharmonic maps by Wang and Zheng \cite{WZ11,WZ12}, for intrinsic Hessian biharmonic maps by Hornung and Moser \cite{HM} and for both cases by Laurain and Rivi\`ere \cite{LR}.
\begin{thm}
	\label{thm:ei}
	Let $u_i$ be a sequence of biharmonic maps from $B^4$ to $N$ satisfying
	\begin{equation*}
		\int_{B_1} \abs{\nabla^2 u_i}^2 +\abs{\nabla u_i}^4 dx <\Lambda
	\end{equation*}
	for some $\Lambda>0$.
	Assume that there is a sequence positive $\lambda_i\to 0$ such that
	\begin{equation*}
		u_i(\lambda_i x)\to \omega
	\end{equation*}
	on any compact set $K\subset \Real^4$, that $u_i$ converges weakly in $W^{2,2}$ to $u_\infty$ and that $\omega$ is the only bubble. Then,
	\begin{equation}\label{eqn:ei}
		\lim_{\delta\to 0} \lim_{R\to \infty} \lim_{i\to \infty} \int_{B_\delta\setminus B_{\lambda_i R}} \abs{\nabla^2 u}^2 +\abs{\nabla u}^4 dx =0.
	\end{equation}
\end{thm}
The proof of Theorem \ref{thm:ei} is completely contained in the proof of Theorem \ref{thm:main}. Our second byproduct is a new proof of the removable singularity. In the theory of harmonic maps, it was first proved by Sacks and Uhlenbeck \cite{SU}. Then it became a special case of H\'elein's regularity theorem \cite{Helein} for weak harmonic maps of two dimensions. It is H\'elein's idea that was generalized to the case of biharmonic maps and it was shown in \cite{Moser} that weak intrinsic Hessian biharmonic map in $W^{2,2}$ is smooth and the case of extrinsic biharmonic maps and intrinsic Laplace biharmonic maps is proved in \cite{W04}. Hence, the removable singularity theorem of biharmonic maps follows as a corollary.

In Section \ref{sec:rem}, we prove the following removable singularity theorem by using an argument similar to the proof of Sacks and Uhlenbeck in \cite{SU}.
\begin{thm}
	Let $u$ be a smooth biharmonic map on $B_1\setminus \set{0}$. If
	\begin{equation*}
		\int_{B_1} \abs{\nabla^2 u}^2 +\abs{\nabla u}^4 dx< +\infty,
	\end{equation*}
	then $u$ can be extended to a smooth biharmonic map on $B_1$.
\end{thm}

Our proof for all three types of biharmonic maps are similar. Hence, we shall present only the complete proof for extrinsic biharmonic maps and show in Section \ref{sec:intrinsic} that why the proof works for other cases as well.

The proof contains two ingredients. The first is a generalization of Qing and Tian's proof of no neck result for harmonic maps \cite{QT}. Precisely, we prove a three circle lemma for biharmonic functions and then show that the lemma holds for some approximate biharmonic functions. For $u_i$ in the neck region,
\begin{equation*}
	u_i(x) -\frac{1}{\abs{\partial B_r}}\int_{\partial B_r} u_i(x) d\sigma
\end{equation*}
will be shown to be approximate biharmonic function in the above sense. Hence, we can argue as in Qing and Tian to see that the tangential derivatives of $u_i$ satisfy some decay estimate.

Next important idea for the proof is a Pohozaev type argument. In the case of harmonic maps, it was first introduced to the study of neck analysis by Lin and Wang \cite{LW} and it says the tangential part of the energy is the same as the radial part. The computation is generalized to biharmonic maps in \cite{HM}, \cite{WZ12} and \cite{LR}. Because the biharmonic maps satisfy a fourth order PDE, boundary terms arise in the computation and the authors of \cite{HM} and \cite{WZ12} managed to show that the boundary terms are small so that they can still compare the tangential energy and the radial energy. In this paper, we make use of this piece of information in a different way. We derive an ordinary differential inequality for the radial part of energy. Thanks to the decay of tangential energy, we can prove that the radial energy decays in a similar way in the neck.

The rest of the paper is organized as follows. In Section \ref{sec:pre}, we set up the notations and recall the $\varepsilon-$regularity theorem and the removable singularity theorem. Moreover, we show that $\lim_{\abs{x}\to \infty} \omega$ exists for a biharmonic map with finite energy from $\Real^4$ to $N$. In Section \ref{sec:bif}, we show the three circle lemma, which is used in Section \ref{sec:tan} to show the exponential decay of tangential energy. The proof of the main theorem is completed in Section \ref{sec:rad} by showing the decay of radial energy. In Section \ref{sec:rem}, we show how to use the method of three circle lemma to give an elementary proof of the removable singularity theorem. In the last section, we indicate why the proofs of this paper work for intrinsic biharmonic maps as well.

\section{Preliminaries}\label{sec:pre}

In this section, we recall some basic results about biharmonic maps in dimension four. For simplicity, we assume that $u$ is a map from $\Omega\subset \Real^4$ into some closed Riemannian manifold $N$ and $N$ is isometrically embedded in $\Real^K$.

An extrinsic biharmonic map is a critical point of $H(u)$, hence it satisfies the Euler-Lagrange equation
\begin{equation}\label{eqn:ELex}
	\triangle^2 u= \triangle( B(u)(\nabla u, \nabla u)) +2 \nabla \cdot \langle \triangle u, \nabla(P(u))\rangle -\langle \triangle (P(u)),\triangle u \rangle.
\end{equation}
Here $P(y):\Real^K\to T_y N$ is the orthogonal projection from $\Real^K$ to the tangent space $T_y N$ and for $X,Y\in T_y N$, $B(y)(X,Y)=-\nabla_X P(y)(Y)$ is the second fundamental form of $N$ as a submanifold in $\Real^K$.
Since we consider smooth biharmonic maps only, it was proved in \cite{W04} that the above equation is equivalent to
\begin{equation*}
\triangle^2 u\perp T_u N.	
\end{equation*}

For intrinsic biharmonic maps, the equations are more complicated. We postpone their discussion to Section \ref{sec:intrinsic}. It suffices to note here that they are scaling invariant.

To study the blow-up, we need

(1) an $\varepsilon-$regularity estimate;

(2) a removable of singularity theorem;

(3) a uniform lower bound on the energy of bubbles.

These results are by now very standard and one can find proofs in Lemma 2.5, Lemma 2.6 and Lemma 2.8 in \cite{HM}, Lemma 5.3 in \cite{S} and Theorem A in \cite{W04}. For completeness, we list them below in the form we need in later sections.

\begin{thm}\label{thm:regularity}($\varepsilon_0-$regularity)
Let $u\in W^{4,p}(B_1)$,$p>1$, be a biharmonic map. There exists $\epsilon_0>0$ such that if $\int_{B_1}|\nabla^2u|^2+|\nabla u|^4dx\leq\epsilon_0$ then
\begin{eqnarray*}
\|u-\overline{u}\|_{W^{4,p}(B_{1/2})}\leq C(\|\nabla^2u\|_{L^2(B_1)}+\|\nabla u\|_{L^4(B_1)}),
\end{eqnarray*}
where $\overline{u}$ is the mean value of $u$ over the unit ball.
\end{thm}
\begin{proof}
The proof is in Appendix.
\end{proof}

\begin{thm}(removable singularity)\label{thm:remove}
	Let $u$ be a biharmonic map defined on $B_1\setminus \set{0}$ into $N$.	Assume that
	\begin{equation*}
		\int_{B_1\setminus \set{0}} \abs{\nabla^2 u}^2 +\abs{\nabla u}^4 dx \leq C,
	\end{equation*}
	then $u$ can be extended smoothly to be a biharmonic map from $B_1$ to $N$.
\end{thm}
It follows from Theorem A of \cite{W04} and Theorem 6.1 and Theorem 6.2 of \cite{Moser}. An argument like Lemma 2.5 in \cite{HM} is needed to show that $u$ is a weak biharmonic map to use the above mentioned theorems.

\begin{thm}
	(energy gap)\label{thm:gap}
	There is some $\varepsilon>0$ depending only on $N$ such that if $u$ is a biharmonic map from $\Real^4$ to $N$ with
	\begin{equation*}
		\int_{\Real^4} \abs{\triangle u}^2 +\abs{\nabla u}^4 dv \leq \varepsilon.
	\end{equation*}
	then $u$ is a constant map.
\end{thm}
This is essentially Lemma 2.8 of \cite{HM} and again the proof works for all three types of biharmonic maps.

To conclude this section, we show how Theorem \ref{thm:main} implies that $\lim_{\abs{x}\to \infty} \omega$ exists, as promised in the introduction.

It suffices to show that for any $\varepsilon>0$, we can find $R$ large such that for any $R'>R$, we have
\begin{equation*}
	osc_{B_{R'}\setminus B_R} \omega <\varepsilon.
\end{equation*}
Since $u_i(\lambda_i x)$ converges strongly to $\omega$ on $B_{R'}\setminus B_R$, we need to show that
\begin{equation*}
	osc_{B_{\lambda_i R'}\setminus B_{\lambda_i R}} u_i <\varepsilon/2.
\end{equation*}
This is a consequence of (\ref{eqn:noneck}) because when $i$ is large $B_{\lambda_i R'}$ is contained in $B_\delta$ for any $\delta>0$.

\section{three circle lemma}\label{sec:bif}
This section consists of two parts. In the first part, we show the three circle lemma for biharmonic functions defined on $B_{r_2}\setminus B_{r_1}\subset \Real^4$. In the second part, we generalize this to some approximate biharmonic functions.

\subsection{biharmonic functions}
Let $f$ be a biharmonic function defined on $B_{r_2}\setminus B_{r_1}\subset \Real^4$. Let $\varphi_n^l (l=1,\cdots,h_n,n=1,2,3,\cdots)$ be the eigen-functions of $S^3$ (excluding constant functions), i.e.
\begin{equation*}
	\triangle_{S^3} \varphi_n^l= -n(n+2) \varphi_n^l.
\end{equation*}
Moreover, we assume that $\set{\varphi_n^l}$ are normalized so that they form an orthonormal basis of $L^2(S^3)$. If we denote the coordinates of $S^3$ by $\theta$, then $\varphi_n^l$ is a function of $\theta$ and $f$ is a function of $r$ and $\theta$, where $r=\abs{x}$.

By separation of variables, we can write
	\begin{eqnarray*}
		f&=&  A_0 +B_0 r^2 + C_0 r^{-2} +D_0 \log r\\
		&&+ \sum_{n=1}^\infty \sum_{l=1}^{h_n} \left( A^l_n r^n + B^l_n r^{-(n+2)} +C^l_n r^{n+2} +D^l_n r^{-n} \right)\varphi_n^l.
	\end{eqnarray*}

For some $L>0$ to be determined later and $i\in \mathbb Z^+$, set
\begin{equation*}
	A_i= B_{e^{-(i-1) L}}\setminus B_{e^{-iL}}
\end{equation*}
and
\begin{equation*}
	F_i(f)=\int_{A_i} \frac{1}{\abs{x}^4} f^2 dx.
\end{equation*}

We will prove the following three circle lemma for biharmonic functions.
\begin{thm}\label{thm:tcfunction}
	There is some universal constant $L>0$. If $f$ is a nonzero biharmonic function defined on $A_{i-1}\cup A_i\cup A_{i+1}$ satisfying
	\begin{equation*}
		\int_{S^3} f(r,\theta) d\theta=0
	\end{equation*}
	for all $r \in [e^{-(i+1)L},e^{-(i-1)L}]$, then
	\begin{equation*}
		2F_i(f)< e^{-L} (F_{i-1}(f)+F_{i+1}(f)).
	\end{equation*}
\end{thm}

The proof is direct computation. First, by our assumption, we may assume that
	\begin{eqnarray*}
		f&=&  \sum_{n=1}^\infty \sum_{l=1}^{h_n} \left( A^l_n r^n + B^l_n r^{-(n+2)} +C^l_n r^{n+2} +D^l_n r^{-n} \right)\varphi_n^l.
	\end{eqnarray*}

Second, since
\begin{equation*}
	F_i(f)=\int_{e^{-i L}}^{e^{-(i-1)L}} \frac{1}{r} \left( \int_{S^3} f^2 d\theta \right) dr
\end{equation*}
and $\set{\varphi_n^l}$ are orthonormal basis, it suffices to prove the theorem for
\begin{equation*}
	f= (A r^n + B r^{-(n+2)} + C r^{n+2} + D r^{-n}) \varphi
\end{equation*}
where $\varphi$ is one of $\set{\varphi_n^l}$.

Finally, we observe that by scaling, it suffices to prove the case $i=0$. Hence, in the following, we assume $f$ is defined on $A_{-1}\cup A_0\cup A_1$.

We note that $F_i$ is a quadratic form of $(A,B,C,D)$. Precisely, we have
\begin{eqnarray*}
	&& \int_{B_{e^{-(i-1)L}}\setminus B_{e^{-iL}}} f^2 \frac{1}{r^4} dx\\
	&=& \int_{e^{-iL}}^{e^{-(i-1)L}} \int_{S^3} (A r^n + B r^{n+2} + C r^{-n} + D r^{-n-2})^2 \varphi^2 \frac{1}{r} dr dV_{S^3} \\
	&=& \int_{e^{-iL}}^{e^{-(i-1)L}} (A r^n + B r^{n+2} + C r^{-n} + D r^{-n-2})^2  \frac{1}{r} dr \\
	&:=& (A,B,C,D) M_i \left(
	\begin{array}[]{c}
		A \\
		B\\
		C\\
		D
	\end{array} \right)
	.
\end{eqnarray*}
Here $M_i$ is a $4\times 4$ matrix.
Direct computation shows
\begin{equation*}
	M_i=
	\left(
	\begin{array}[]{cccc}
		g(2n) & g(2n+2) & g(0) & g(-2) \\
		g(2n+2) & g(2n+4) & g(2) & g(0) \\
		g(0) & g(2) & g(-2n) & g(-2n-2) \\
		g(-2) & g(0) & g(-2n-2) & g(-2n-4)
	\end{array}
	\right)
\end{equation*}
where
\begin{equation*}
	g(\beta)=\int_{e^{-iL}}^{e^{-(i-1)L}}r^{\beta-1} dr = \left\{
		\begin{array}[]{ll}
			L & \quad \beta=0 \\
			\frac{1}{\beta}e^{-i \beta L}(e^{\beta L}-1) & \quad \beta\ne 0
		\end{array}
		\right.
\end{equation*}

By the above discussion, to prove the three circle lemma for biharmonic functions, it suffices to show that the matrix
\begin{equation*}
	e^{-L}(M_{-1}+M_{1})-2M_0
\end{equation*}
is positive definite for some universal constant $L>0$.

It turns out that we can choose the universal constant $L=3$. For $n=1$, we can check by hand (or by computer software) that the above $4$ by $4$ matrix is positive definite. In the following proof, keep in mind that $L=3$ and $n\geq 2$ and we need this to justify certain inequalities. (We will not mention this fact every time we use it.)

We write this matrix in the form of
\begin{equation}\label{eqn:matrix}
	\left(
	\begin{array}[]{cc}
		\mathcal A & \mathcal B \\
		\mathcal B^T & \mathcal C
	\end{array}
	\right).
\end{equation}
Here
\begin{equation*}
	\mathcal A=\left(
	\begin{array}[]{cc}
		\frac{(e^{2nL}-1)(e^{(2n-1)L}+e^{-(2n+1)L}-2)}{2n} & \frac{(e^{(2n+2)L}-1)(e^{(2n+1)L}+e^{-(2n+3)L}-2)}{2n+2}\\
		\frac{(e^{(2n+2)L}-1)(e^{(2n+1)L}+e^{-(2n+3)L}-2)}{2n+2} & \frac{(e^{(2n+4)L}-1) ( e^{(2n+3)L}+e^{-(2n+5)L}-2)}{2n+4}
	\end{array}
	\right),
\end{equation*}

\begin{equation*}
	\mathcal B=\left(
	\begin{array}[]{cc}
		2L(e^{-L}-1) & - \frac{1}{2}(e^{-2L}-1) (e^{-3L}+e^L -2) \\
		\frac{1}{2} (e^{2L}-1)(e^L +e^{-3L}-2) & 2L (e^{-L}-1)
	\end{array}
	\right)
\end{equation*}
and
\begin{equation*}
	\mathcal C=\left(
	\begin{array}[]{cc}
		-\frac{(e^{-2nL}-1)(e^{-(2n+1)L}+e^{(2n-1)L}-2)}{2n} & -\frac{(e^{-(2n+2)L}-1)(e^{-(2n+3)L}+e^{(2n+1)L}-2) }{2n+2}\\
		-\frac{(e^{-(2n+2)L}-1)(e^{-(2n+3)L}+e^{(2n+1)L}-2)}{2n+2} & -\frac{(e^{-(2n+4)L}-1) ( e^{-(2n+5)L}+e^{(2n+3)L}-2)}{2n+4}
	\end{array}
	\right).
\end{equation*}

In order to show that (\ref{eqn:matrix}) is positive definite, we show that both $\mathcal A$ and $\mathcal C$ are positive definite and they dominate $\mathcal B$ so that the whole matrix is positive definite. More precisely, we consider
\begin{eqnarray*}
	&&(x^T,y^T)
	\left(
	\begin{array}[]{cc}
		\mathcal A & \mathcal B \\
		\mathcal B^T & \mathcal C
	\end{array}
	\right)
	\left(
	\begin{array}[]{c}
		x\\
		y
	\end{array}
	\right)\\
	&=& x^T\mathcal A x + x^T \mathcal B y + y^T \mathcal B^T x + y^T \mathcal C y.
\end{eqnarray*}

Let $m$ be the largest coefficient in $\mathcal B$, then
\begin{equation}
	\abs{x^T \mathcal B y} + \abs{y^T \mathcal B^T x}\leq 2 m \abs{x} \abs{y}
	\label{eqn:B}
\end{equation}
Denoting the small eigenvalues of $\mathcal A$ and $\mathcal C$ by $\lambda$ and $\mu$ respectively, we have
\begin{equation*}
	x^T\mathcal Ax\geq \lambda\abs{x}^2,\qquad y^T \mathcal C y \geq \mu \abs{y^2}.
\end{equation*}
Hence, it suffices to show that
\begin{equation*}
	\lambda \abs{x}^2 -2m \abs{x}\abs{y} +\mu \abs{y}^2\geq 0
\end{equation*}
for all $x$ and $y$. This is equivalent to
\begin{equation}
m^2-\lambda \mu<0.
	\label{eqn:final}
\end{equation}

For a two by two symmetric matrix
\begin{equation*}
	\left(
	\begin{array}[]{cc}
		a & b \\
		b &  c
	\end{array}
	\right),
\end{equation*}
the smaller eigenvalue is given by
\begin{equation*}
	\lambda=\frac{4(ac-b^2)}{2(a+c+\sqrt{(a-c)^2+4b^2})}.
\end{equation*}

For matrix $\mathcal A$, since $L=3$, then it is obvious that $c>b>a$ for all $n$. With this in mind, we claim that
\begin{equation}\label{eqn:acb}
	\lambda\geq \frac{ac-b^2}{3 c}.
\end{equation}
In fact, we shall prove below that $ac-b^2>0$ and the claim follow from
\begin{equation*}
	a+c+\sqrt{(a-c)^2+4b^2}\leq 6 c.
\end{equation*}
To see $ac-b^2>0$, we compute
\begin{eqnarray*}
	a&=& \frac{1}{2n}(e^{2nL}-1)(e^{(2n-1)L}+e^{-(2n+1)L}-2) \\
	&=& \frac{1}{2n} \left(  e^{(4n-1) L}- 2 e^{2nL}- e^{(2n-1)L} + e^{-L}-e^{-(2n+1)L}+2 \right) \\
	&\geq& \frac{1}{2n} \left(e^{(4n-1) L}- 2 e^{2nL}- e^{(2n-1)L} \right) \\
	&\geq& \frac{1}{2n} \left( e^{(4n-1)L} -3 e^{2n L} \right).
\end{eqnarray*}
Here we used $e^{-2L}-e^{-(2n+1)L}+2\geq 0$ for all $n$ and $L=3$.

Similarly,
\begin{eqnarray*}
	c&=& \frac{1}{2n+4}(e^{(2n+4)L}-1) ( e^{(2n+3)L}+e^{-(2n+5)L}-2) \\
	&=& \frac{1}{2n+4} \left( e^{(4n+7)L}+ e^{-L}-2 e^{(2n+4)L}-e^{(2n+3)L}-e^{-(2n+5)L}+2 \right) \\
	&\geq& \frac{1}{2n+4} \left( e^{(4n+7)L}-3 e^{(2n+4)L} \right).
\end{eqnarray*}

We also need upper bound of $b$ and $c$.
\begin{eqnarray*}
	c&=& \frac{1}{2n+4}(e^{(2n+4)L}-1) ( e^{(2n+3)L}+e^{-(2n+5)L}-2) \\
	&\leq& \frac{1}{2n+4} e^{(4n+7)L}
\end{eqnarray*}
and
\begin{eqnarray*}
	b&=&  \frac{1}{2n+2}(e^{(2n+2)L}-1)(e^{(2n+1)L}+e^{-(2n+3)L}-2) \\
	&\leq& \frac{1}{2n+2}  e^{(4n+3)L} .
\end{eqnarray*}

\begin{eqnarray*}
	ac-b^2 &\geq& \frac{1}{(2n)(2n+4)} \left( e^{(4n-1)L}-3 e^{2n L} \right) \left( e^{(4n+7)L}- 3 e^{(2n+4)L} \right) \\
	&& -\frac{1}{(2n+2)^2} e^{(8n+6)L}  \\
	&\geq& \frac{1}{(2n)(2n+4)} \left( e^{(8n+6)L}-6 e^{(6n+7)L} \right) -\frac{1}{(2n+2)^2}  e^{(8n+6)L} \\
	&\geq& \frac{4}{(2n)(2n+2)^2(2n+4)}e^{(8n+6)L}-\frac{4}{n^2} e^{(6n+7)L}.
\end{eqnarray*}

\begin{eqnarray*}
	\lambda&\geq& \frac{1}{12 n(n+1)^2} e^{(4n-1)L} -\frac{8}{n} e^{2nL} \\
\end{eqnarray*}
For $L=3$ and all $n\geq 2$, we have
\begin{eqnarray}\label{eqn:lambda}
	\lambda&\geq& \frac{1}{24 n(n+1)^2} e^{(4n-1)L}
\end{eqnarray}
for all $n$.

Now, we repeat the above argument for
\begin{equation*}
	C=\left(
	\begin{array}[]{cc}
		-\frac{e^{-2nL}-1}{2n}(e^{-(2n+1)L}+e^{(2n-1)L}-2) & -\frac{e^{-(2n+2)L}-1}{2n+2}(e^{-(2n+3)L}+e^{(2n+1)L}-2) \\
		-\frac{e^{-(2n+2)L}-1}{2n+2}(e^{-(2n+3)L}+e^{(2n+1)L}-2) & -\frac{e^{-(2n+4)L}-1}{2n+4} ( e^{-(2n+5)L}+e^{(2n+3)L}-2)
	\end{array}
	\right).
\end{equation*}
We still use $a,b$ and $c$ to denote the coefficients in $\mathcal C$ and we still have $a<b<c$.
Similar computation shows that
\begin{eqnarray*}
	a&=& \frac{1}{2n}\left( 1-e^{-2nL} \right)\left( e^{-(2n+1)L}+ e^{(2n-1)L}-2 \right) \\
	&=& \frac{1}{2n} \left( e^{(2n-1)L}+ e^{-(2n+1)L}-2 - e^{-(4n+1)L} -e^{-L}+2 e^{-2nL} \right) \\
	&\geq&\frac{1}{2n} \left( e^{(2n-1)L}-4 \right),
\end{eqnarray*}

\begin{eqnarray*}
	c&=& \frac{1}{2n+4} \left( 1-e^{-(2n+4)L}  \right)( e^{-(2n+5)L}+e^{(2n+3)L}-2) \\
	&=& \frac{1}{2n+4} \left( e^{(2n+3)L}+ e^{-(2n+5)L}-2 -e^{-(4n+9)L}- e^{-L} +2 e^{-(2n+4)L} \right),
\end{eqnarray*}
\begin{equation*}
	\frac{1}{2n+4}\left( e^{(2n+3)L}-4 \right)\leq c\leq \frac{1}{2n+4} e^{(2n+3)L},
\end{equation*}

\begin{eqnarray*}
	b&=& \frac{1}{2n+2}\left( 1-e^{-(2n+2)L} \right)\left( e^{-(2n+3)L}+ e^{(2n+1)L}-2 \right) \\
	&\leq & \frac{1}{2n+2} e^{(2n+1)L},
\end{eqnarray*}
and
\begin{eqnarray*}
	ac-b^2 &\geq& \frac{1}{(2n)(2n+4)}\left( e^{(2n-1)L}-4 \right) \left( e^{(2n+3)L}-4 \right) - \frac{1}{(2n+2)^2} e^{(4n+2)L} \\
	&\geq& \frac{4}{(2n)(2n+4)(2n+2)^2} e^{(4n+2)L} -\frac{8}{(2n)(2n+4)} e^{(2n+3)L}.
\end{eqnarray*}

In summary,
\begin{eqnarray*}
	\mu&\geq& \frac{ac-b^2}{3c} \\
	&\geq& \frac{2n+4}{3}e^{-(2n+3)L} \left( \frac{4}{(2n)(2n+4)(2n+2)^2} e^{(4n+2)L} -\frac{8}{(2n)(2n+4)} e^{(2n+3)L} \right)  \\
	&\geq& \frac{1}{12n(n+1)^2} e^{(2n-1)L}.
\end{eqnarray*}

Finally, by the formula of $\mathcal B$, we see $m\leq \frac{1}{2}e^{3L}$. For $n\geq 2$, $m^2<\lambda\mu$ is implied by
\begin{equation*}
	\frac{1}{4}e^{6L}<\frac{1}{24n(n+1)}e^{(4n-1)L} \frac{1}{12n(n+1)^2}e^{(2n-1)L},
\end{equation*}
which is true if $L=3$.
This concludes the proof of Theorem \ref{thm:tcfunction}.

\subsection{approximate biharmonic functions}
\label{subsec:approx}
By an approximate biharmonic function, we mean a smooth solution $u$ defined on $B_{r_2}\setminus B_{r_1}$ satisfying
\begin{eqnarray}\label{eqn:approx}
\triangle^2u (r,\theta)&=& a_1\nabla\triangle u+a_2\nabla^2 u+a_3\nabla u+a_4 u \\ \nonumber
&& +
\frac{1}{\abs{\partial B_r}}\int_{\partial B_r}{b_1\nabla\triangle u+b_2\nabla^2 u+b_3\nabla u+b_4 u}.
\end{eqnarray}
Here $a_i$ and $b_i$ are smooth functions, which will be small in the sense that for some small $\eta>0$,
\begin{equation}\label{eqn:small}
	\abs{a_j}+\abs{b_j}\leq \frac{\eta}{\abs{x}^j}\quad \mbox{on} \quad B_{r_2}\setminus B_{r_1}.
\end{equation}
Sometimes, to emphasize (\ref{eqn:small}), we say the function is an $\eta-$approximate biharmonic function.
\begin{rem}
	Note that if $u$ is $\eta-$approximate biharmonic function on $B_{r_2}\setminus B_{r_1}$, $u(\frac{x}{\lambda})$ is also $\eta-$approximate biharmonic on $B_{\lambda r_2}\setminus B_{\lambda r_1}$.
\end{rem}

(\ref{eqn:approx}) is not the usual type of PDE because of the integral term. However, we still have the interior $L^p$ estimate.

\begin{lem}
	\label{lem:lp}
	Suppose that $u: B_4\setminus B_1\to \Real$ is an approximate biharmonic function and
	\begin{equation*}
		\sum_{i=1}^4 \norm{a_i}_{L^\infty} +\norm{b_i}_{L^\infty}\leq C.
	\end{equation*}
	Then we have
	\begin{equation*}
		\norm{u}_{W^{4,p}(B_3\setminus B_2)}\leq C \norm{u}_{L^p(B_4\setminus B_1)}.
	\end{equation*}
\end{lem}

There is a similar lemma for approximate harmonic function in \cite{Liu} (see Lemma 3.1).

\begin{proof}
	For $0< \sigma<1$, set $A_\sigma= B_{3+\sigma}\setminus B_{2-\sigma}$ and $A'_\sigma= B_{3+\frac{1+\sigma}{2}}\setminus B_{2-\frac{1+\sigma}{2}}$. Let $\varphi$ be a cut-off function supported in $A'_\sigma$ satisfying: (1) $\varphi\equiv 1$ in $A_\sigma$; (2) $\abs{\nabla^j \varphi}\leq \frac{c}{(1-\sigma)^j}$ for $j=1,2,3,4$ and some universal constant $c$; (3) $\varphi$ is a function of $\abs{x}$.

Computing directly, we have
\begin{eqnarray*}
\triangle^2(\varphi u)&=&\triangle(\varphi\triangle u+2\nabla\varphi\nabla u+u\triangle\varphi)\\
&=&\varphi\triangle^2 u+4\nabla\triangle u\nabla\varphi+4\nabla^2u\nabla^2\varphi+2\triangle u\triangle\varphi+4\nabla\triangle \varphi\nabla u+\triangle^2\varphi u\\
&=&  \varphi a_1 \nabla\triangle u + \varphi  a_2 \nabla^2  u + \varphi a_3 \nabla  u +\varphi  a_4 u \\
&& + \varphi \frac{1}{\abs{\partial B_r}}\int_{\partial B_r} b_1 \nabla\triangle u + b_2 \nabla^2 u + b_3 \nabla  u + b_4  u \\
&& + 4\nabla\triangle u\nabla\varphi+4\nabla^2u\nabla^2\varphi+2\triangle u\triangle\varphi+4\nabla\triangle \varphi\nabla u+\triangle^2\varphi u.
\end{eqnarray*}
Next, we estimate the $L^p (p>1)$ norm of the right hand side of the above equation. By our choice of $\varphi$ and the assumption of $a_1$, we have
\begin{equation*}
	\norm{\nabla \triangle u \nabla \varphi}_{L^p(A'_\sigma)} + \norm{\varphi a_1 \nabla \triangle u}_{L^p(A'_\sigma)} \leq \frac{C}{1-\sigma} \norm{\nabla^3 u}_{L^p(A'_\sigma)}
\end{equation*}
Moreover, Jensen's inequality implies that
\begin{eqnarray*}
	&& \int_{A'_\sigma} \frac{\varphi^p}{\abs{\partial B_r}^p} \left( \int_{\partial B_r} b_1 \nabla \triangle u \right)^p dx \\
	&\leq& \int_{A'_{\sigma}} \varphi^p \frac{1}{\abs{\partial B_r}}\left( \int_{\partial B_r} \abs{b_1 \nabla \triangle u}^p \right) dx \\
	&\leq & C \int_{A'_\sigma} \abs{\nabla^3 u}^p dx.
\end{eqnarray*}
Similar argument applies to the remaining terms and gives an estimate of $L^p$ norm of $\triangle^2(\varphi u)$, by which the $L^p$ estimate of bi-Laplace operator implies
\begin{equation*}
	\norm{\varphi u}_{W^{4,p}(A'_\sigma)}\leq C \left( \frac{\norm{\nabla^3 u}_{L^p(A'_\sigma)}}{1-\sigma} +\frac{\norm{\nabla^2 u}_{L^p(A'_\sigma)}}{(1-\sigma)^2} +\frac{\norm{\nabla u}_{L^p(A'_\sigma)}}{(1-\sigma)^3 } +\frac{\norm{u}_{L^p(A'_\sigma)}}{(1-\sigma)^4}  \right).
\end{equation*}
In particular, we have
\begin{eqnarray*}
	(1-\sigma)^4 \norm{\nabla^4 u}_{L^p(A_\sigma)} & \leq&  C \left( (1-\sigma)^3 {\norm{\nabla^3 u}_{L^p(A'_\sigma)}} + (1-\sigma)^2 {\norm{\nabla^2 u}_{L^p(A'_\sigma)}}\right.  \\
	&& \left. + (1-\sigma) {\norm{\nabla u}_{L^p(A'_\sigma)}} + {\norm{u}_{L^p(A'_\sigma)}}  \right).
\end{eqnarray*}
By setting
\begin{equation*}
	\Psi_j=\sup_{0\leq \sigma\leq 1} (1-\sigma)^j \norm{\nabla^j u}_{L^p (A_\sigma)}
\end{equation*}
and noting that
\begin{equation*}
	A'_\sigma=A_{\frac{1+\sigma}{2}}\quad \mbox{and}\quad 1-\sigma= 2 (1-\frac{1+\sigma}{2}),
\end{equation*}
we obtain
\begin{equation}\label{eqn:seminorm}
	\Psi_4\leq C(\Psi_3 +\Psi_2 +\Psi_1 +\Psi_0).
\end{equation}

We claim that for $j=1,2,3$, the following interpolation inequality holds for any $\epsilon>0$,
\begin{equation*}
	\Psi_j\leq \epsilon^{4-j} \Psi_4 +\frac{C}{\epsilon^j} \Psi_0.
\end{equation*}
In fact, by the definition of $\Psi_j$, for any $\gamma>0$, there is $\sigma_\gamma\in [0,1]$ such that
\begin{eqnarray*}
	\Psi_j &\leq& (1-\sigma_j)^j \norm{\nabla^j u}_{L^p(A_{\sigma_\gamma})}+ \gamma \\
	&\leq& \epsilon^{4-j}(1-\sigma_\gamma)^4 \norm{\nabla^4 u}_{L^p(A_{\sigma_\gamma})} +\frac{C}{\epsilon^j}\norm{u}_{L^p(A_{\sigma_\gamma})} +\gamma \\
	&\leq& \epsilon^{4-j} \Psi_4 + \frac{C}{\epsilon^j} \Psi_0 +\gamma .
\end{eqnarray*}
Here we used the interpolation inequality
\begin{equation*}
	\norm{\nabla^j u}_{L^p(A_{\sigma_\gamma})}\leq \eta^{4-j}\norm{\nabla^4 u}_{L^p(A_{\sigma_\gamma})}+\frac{C}{\eta^j} \norm{u}_{L^p(A_{\sigma_\gamma})}
\end{equation*}
with $\eta= \epsilon (1-\sigma_\gamma)$. We remark that the constant in the above interpolation inequality is independent of $\sigma_\gamma\in [0,1]$ (see the proof of Lemma 5.6 in \cite{adams}).
By sending $\gamma$ to $0$ and choosing small $\epsilon$, we obtain from (\ref{eqn:seminorm})
\begin{equation*}
	\Psi_4\leq C\Psi_0,
\end{equation*}
from which our lemma follows.
\end{proof}

Now, we prove the three circle lemma for approximate biharmonic functions. For two positive integers $l_1$ and $l_2$ with ($l_1>l_2$), set
\begin{equation*}
	\Sigma=\bigcup_{i=l_2}^{l_1} A_i
\end{equation*}
and recall that $A_i= B_{e^{-(i-1)L}}\setminus B_{e^{-iL}}$ where $L$ is the universal constant in Theorem \ref{thm:tcfunction}.

\begin{thm}
	\label{thm:tcapprox}
	There is some constant $\eta_0>0$ such that the following is true. Assume that $u:\Sigma\to \Real^K$ is an $\eta_0-$approximate biharmonic function in the sense of (\ref{eqn:approx}) and that
	\begin{equation}\label{eqn:notheta}
		\int_{\partial B_r} u d\theta=0
	\end{equation}
	for $r\in [e^{-l_1 L}, e^{-(l_2-1)L}]$. Then for any integer $i$ with $l_1>i>l_2$, we have

	(a) if $F_{i+1}(u)\leq e^{-L} F_i(u)$, then $F_i(u)\leq e^{-L} F_{i-1}(u)$;

	(b) if $F_{i-1}(u)\leq e^{-L} F_i(u)$, then $F_i(u)\leq e^{-L} F_{i+1}(u)$;

	(c) either $F_i(u)\leq e^{-L} F_{i-1}(u)$, or $F_i(u)\leq e^{-L} F_{i+1}(u)$.
\end{thm}

\begin{proof}
	The exact value of $i$ does not matter, because $F_i$ is invariant under scaling. Hence, we consider only the case of $i=2$. Assume the theorem is not true. We have a sequence of $\eta_k\to 0$ and a sequence of $u_k$ defined on $A_1\cup A_2\cup A_3$ satisfying
\begin{eqnarray}\label{eqn:kk}
	\triangle^2u_k (r,\theta)&=& a_{k1}\nabla\triangle u_k+a_{k2}\nabla^2 u_k+a_{k3}\nabla u_k+a_{k4} u_k \\ \nonumber
&& +
\frac{1}{\abs{\partial B_r}}\int_{\partial B_r}{b_{k1}\nabla\triangle u_k+b_{k2}\nabla^2 u_k+b_{k3}\nabla u_k+b_{k4} u_k}
\end{eqnarray}
with
\begin{equation}\label{eqn:kksmall}
	\abs{a_{ki}}+\abs{b_{ki}}\leq \eta_k \quad \mbox{on} \quad A_1\cup A_2\cup A_3.
\end{equation}
By taking subsequence, we assume that one of (a), (b) and (c) is not true for $u_k$. If (a) is not true, then we have
\begin{equation*}
	F_2(u_k)\geq e^{L} F_3(u_k) \quad\mbox{and}\quad F_2(u_k)> e^{-L}F_1(u_k).
\end{equation*}
If (b) is not true, then
\begin{equation*}
	F_2(u_k)\geq e^{L} F_1(u_k) \quad\mbox{and}\quad F_2(u_k)> e^{-L}F_3(u_k).
\end{equation*}
If (c) is not true, then
\begin{equation*}
	F_2(u_k)> e^{-L} \max \{F_1(u_k), F_3(u_k)\}.
\end{equation*}
In any case, we control $F_1(u_k)$ and $F_3(u_k)$ by $F_2(u_k)$. Multiplying by a constant to $u_k$ if necessary, we assume that $F_2(u_k)=1$ for all $k$. The above discussion shows that
\begin{equation*}
	\norm{u_k}_{L^2(A_1\cup A_2\cup A_3)}\leq C.
\end{equation*}
Lemma \ref{lem:lp} shows that (by passing to a subsequence) we have
\begin{eqnarray*}
u_k\rightharpoonup u\quad && weakly \quad in \quad L^2(A_{1}\cup A_2\cup A_3),\\
u_k\rightarrow u \quad
&& strongly \quad in \quad L^2(A_2).
\end{eqnarray*}
By (\ref{eqn:kk}) and (\ref{eqn:kksmall}), we know that $u$ is a nonzero biharmonic function defined on $A_1\cup A_2\cup A_3$ satisfying (\ref{eqn:notheta}). Theorem \ref{thm:tcfunction} implies that
\begin{equation}\label{eqn:good}
	2F_2(u)< e^{-L} (F_1(u)+F_3(u)).
\end{equation}
If (c) does not hold for $u_k$, we have
\begin{equation*}
	2F_2(u_k)\geq e^{-L}(F_1(u_k)+F_3(u_k)).
\end{equation*}
By the strong convergence of $u_k$ in $L^2(A_2)$ and weak convergence in $L^2(A_1\cup A_2\cup A_3)$, we have
\begin{equation*}
	2F_2(u)\geq e^{-L}(F_1(u)+F_3(u)),
\end{equation*}
which is a contradiction to (\ref{eqn:good}). Similar argument works for other cases.
\end{proof}
\section{decay of tangential energy}\label{sec:tan}

In this section, we assume that $u_i$ is a sequence of biharmonic maps defined on $B_1\subset \Real^4$, which blows up at $0$, converges to a weak limit $u_\infty$ and for some sequence $\lambda_i\to 0$, we obtain the only bubble map
\begin{equation*}
	\omega(x)=\lim_{i\to \infty} u_i(\lambda_i x).
\end{equation*}

The neck region is $\Sigma=B_\delta\setminus B_{\lambda_i R}$ for small $\delta$ and large $R$. Assume without loss of generality that
\begin{equation*}
	\Sigma=\bigcup_{l=l_0}^{l_i} A_l
\end{equation*}
for $A_l=B_{e^{-(l-1)L}}\setminus B_{e^{-lL}}$ and $l_0<l_i$. Note that $l_i$ is related to $\lambda_i$ and changes with $i$.

As in \cite{DT}, for any $\varepsilon>0$, we may assume by choosing $\delta$ small and $R$ large, that
\begin{equation}\label{eqn:DT}
	\int_{A_l} \abs{\nabla^2 u_i}^2 +\abs{\nabla u_i}^4 <\varepsilon^4<\varepsilon_0,
\end{equation}
for $l=l_0,\cdots,l_i$ and sufficiently large $i$.
Since our aim is to prove
\begin{equation*}
	\lim_{\delta\to 0}\lim_{R\to \infty}\lim_{i\to \infty} \mbox{osc}_{B_\delta\setminus B_{\lambda_i R}} u_i=0,
\end{equation*}
it suffices to show that for any $\varepsilon>0$ and let $\delta$ and $R$ be determined as above and show
\begin{equation*}
	\mbox{osc}_{B_\delta\setminus B_{\lambda_i R}} u_i<C\varepsilon
\end{equation*}
for $i$ sufficiently large.

Set
\begin{equation*}
	u_i^*(r)=\frac{1}{\abs{\partial B_r}} \int_{\partial B_r} u(r,\theta) d\sigma.
\end{equation*}
By scaling and Poincar\'e inequality, we see
\begin{equation}\label{eqn:poincare}
	\int_{A_l} \frac{1}{\abs{x}^4}\abs{u_i-u_i^*}^2 dx \leq C \varepsilon^2.
\end{equation}

\begin{lem}\label{lem:beapprox}
	There exists some $\varepsilon_1>0$ that if $\varepsilon<\varepsilon_1$ in (\ref{eqn:DT}), $w_i=u_i-u_i^*$ is an $\eta_0-$approximate biharmonic function in the sense of (\ref{eqn:approx}). Here $\eta_0$ is the constant in Theorem \ref{thm:tcapprox}.
\end{lem}

\begin{proof}
	For simplicity, we omit the subscript $i$.
Recall that the Euler-Lagrange equation of biharmonic map is
\begin{eqnarray}\label{eqn:euler}
	\triangle^2 u &=& \alpha_1(u) \nabla \triangle u\# \nabla u+ \alpha_2(u) \nabla^2  u\# \nabla^2 u \\ \nonumber
&&	+ \alpha_3(u) \nabla^2 u \# \nabla u \# \nabla u + \alpha_4(u)\nabla u \# \nabla u \#\nabla u \# \nabla u.
\end{eqnarray}
Here $\alpha_i(u)$ is a smooth function of $u$ and $\#$ is some 'product' for which we are only interested in the properties such as
\begin{equation*}
	\abs{\nabla \triangle u\# \nabla u}\leq C \abs{\nabla \triangle u} \abs{\nabla u}.
\end{equation*}
Since $\triangle = \frac{\partial^2}{\partial r^2} + \frac{3}{r}\pfrac{}{r} +\frac{1}{r^2}\triangle_{S^3}$ and $\int_{S^3} \triangle f d\theta=0$ for any $f$, we have
\begin{eqnarray*}
	\triangle^2 u^*(r)&=& \frac{1}{\abs{\partial B_r}} \int_{\partial B_r} \triangle^2 u d\sigma \\
	&=& \frac{1}{\abs{\partial B_r}}\int_{\partial B_r} \alpha_1(u) \nabla \triangle u\# \nabla u+ \alpha_2(u) \nabla^2  u\# \nabla^2 u \\
&&	+ \alpha_3(u) \nabla^2 u \# \nabla u \# \nabla u + \alpha_4(u)\nabla u \# \nabla u \#\nabla u \# \nabla u d\sigma \\
&=& I + II + III + IV.
\end{eqnarray*}

Computing directly, we get
\begin{eqnarray*}
	I &=& \frac{1}{\abs{\partial B_r}} \int_{B_r} \alpha_1(u) \nabla \triangle u \# \nabla u -\alpha_1(u^*) \nabla \triangle u \# \nabla u \\
	&&  + \alpha_1(u^*) \nabla \triangle u \# \nabla u -\alpha_1(u^*) \nabla \triangle u^* \# \nabla u \\
	&& + \alpha_1(u^*) \nabla \triangle u^* \#\nabla u -\alpha_1(u^*) \nabla \triangle u^* \# \nabla u^* d\sigma \\
	&& + \alpha_1(u^*) \nabla \triangle u^* \# \nabla u^* \\
	&=& \frac{1}{\abs{\partial B_r}} \int_{\partial B_r} \beta_4[u] (u-u^*) + \beta_1[u] \nabla \triangle (u-u^*)\\
	&& + \beta_3[u] \nabla (u-u^*) d\sigma + \alpha_1(u^*) \nabla \triangle u^* \# \nabla u^*.
\end{eqnarray*}
Here $\beta_i[u]$ is some expression depending on $u$, $u^*$ and their derivatives. Those $\beta_i$'s may differ from line to line in the following. However, thanks to Theorem \ref{thm:regularity}, we have
\begin{equation*}
	\abs{\beta_i}(x)\leq \frac{\eta_0}{\abs{x}^i}
\end{equation*}
if $\varepsilon$ in (\ref{eqn:DT}) is smaller than some $\varepsilon_1$. We shall require the above holds for all $\beta_i$ and $\beta'_i$ below by asking $\varepsilon_1$ to be smaller and smaller.

The same computation gives
\begin{eqnarray*}
	II&=& \frac{1}{\abs{\partial B_r}} \int_{B_r} \beta_4[u] (u-u^*) + \beta_2[u] \nabla^2 (u-u^*) d\sigma + \alpha_2(u) \nabla^2 u^* \# \nabla^2 u^*,
\end{eqnarray*}
\begin{eqnarray*}
	III&=&  \frac{1}{\abs{\partial B_r}} \int_{\partial B_r}\beta_4[u] (u-u^*) + \beta_2[u] \nabla^2 (u-u^*) + \beta_3[u] \nabla (u-u^*) d\sigma \\
	&& + \alpha_3(u^*) \nabla^2 u^* \# \nabla u^* \# \nabla u^*
\end{eqnarray*}
and
\begin{eqnarray*}
	IV &=& \frac{1}{\abs{\partial B_r}} \int_{\partial B_r} \beta_4[u] (u-u^*) + \beta_3[u] \nabla (u-u^*) d\sigma + \alpha_4(u) \nabla u^* \#\nabla u^* \#\nabla u^* \#\nabla u^*.
\end{eqnarray*}
In summary, $u^*$ satisfies an equation similar to (\ref{eqn:euler}) except an error term of the form
\begin{equation*}
	\frac{1}{\abs{\partial B_r}} \int_{\partial B_r} \beta_1[u] \nabla \triangle w + \beta_2[u] \nabla^2 w + \beta_3[u] \nabla w + \beta_4[u] w d\sigma.
\end{equation*}
Subtract the equation of $u^*$ with (\ref{eqn:euler}) and handle the terms like $\alpha_1(u)\nabla \triangle u \# \nabla u- \alpha_1(u^*) \nabla \triangle u^* \# \nabla u^*$ as before to get
\begin{eqnarray}\label{eqn:w}
	\triangle^2 w &=& \beta'_1[u] \nabla \triangle w + \beta'_2 [u] \nabla^2  w
		+ \beta'_3[u] \nabla w  + \beta'_4[u] w \\ \nonumber
		&& +
	\frac{1}{\abs{\partial B_r}} \int_{\partial B_r} \beta_1[u] \nabla \triangle w + \beta_2[u] \nabla^2 w + \beta_3[u] \nabla w + \beta_4[u] w d\sigma.
\end{eqnarray}
This concludes the proof of the lemma.
\end{proof}

Now we apply Theorem \ref{thm:tcapprox} to the function $w_i$.
\begin{lem}
	For sufficiently small $\varepsilon>0$, we have
\begin{equation*}
	F_l(w_i)\leq  C\varepsilon^2 \left( e^{-L (l-l_0)}+ e^{-L(l_i-l)} \right).
\end{equation*}
\end{lem}
\begin{proof}
	We start from $l=l_0+1$ and consider $A_{l-1}\cup A_l \cup A_{l+1}$. By (c) of Theorem \ref{thm:tcapprox}, either $F_{l}(w_i)\leq e^{-L} F_{l-1}(w_i)$, or $F_l(w_i)\leq e^{-L} F_{l+1}(w_i)$. If the first case occurs, we move on by adding $l$ by $1$ and repeat the same discussion. The argument above stops if (1) $l=l_i-1$ so that we can not increase $l$ any more, or (2) we find some $l'$ so that (by (b) of Theorem \ref{thm:tcapprox})
	\begin{equation*}
		F_{l-1}(w_i)\geq e^L F_l(w_i) \quad \mbox{for} \quad l=l_0+1,\dots,l'
	\end{equation*}
	and
	\begin{equation*}
		F_{l+1}(w_i)\geq e^L F_l(w_i) \quad \mbox{for} \quad l=l',\dots,l_i-1.
	\end{equation*}
	One may check that the lemma is true in either case, because we have
	\begin{equation*}
		F_{l_0}(w_i), F_{l_i}(w_i)\leq C\varepsilon^2
	\end{equation*}
	by (\ref{eqn:poincare}).
\end{proof}

We conclude this section by showing a pointwise decay estimate.
\begin{lem}\label{lem:point}
\begin{equation}\label{eqn:tangentdecay}
	\max_{A_l} {\abs{x}^p} \abs{\partial_r^p \tilde{\nabla}^q_{S^3} u_i}\leq  C\varepsilon \left( e^{-\frac{L}{2} (l-l_0)}+ e^{-\frac{L}{2}(l_i-l)} \right)
\end{equation}
for all integers $p+q\leq 3$ and $q\geq 1$. Or equivalently, by setting $r=e^t$ and taking $u$ as a function of $(t,\theta)$, we have
\begin{equation}
	\abs{\partial_t ^p \tilde{\nabla}_{S^3}^q u_i} (t,\theta)\leq C\varepsilon\left( e^{-\frac{1}{2} (\log \delta- t)} + e^{-\frac{1}{2}(t-\log \lambda_i R)} \right).
	\label{eqn:weuse}
\end{equation}
\end{lem}
Here $\tilde{\nabla}_{S^3}$ is the partial derivative with respect to $\theta$ in polar coordinates $(r,\theta)$, or equivalently, the gradient operator on the unit sphere $S^3$.

\begin{proof}
Setting
\begin{equation*}
	\tilde{w}(x)= w_i(e^{-(l-1)L} x),
\end{equation*}
we estimate
\begin{equation*}
	\norm{\tilde{w}}_{L^2(A_1)}^2 \leq C F_l(w_i)\leq C\varepsilon^2 \left( e^{-L(l-l_0)}+ e^{-L(l_i-l)} \right).
\end{equation*}
Similarly, $\norm{\tilde{w}}_{L^2(A_0\cup A_1\cup A_2)}$ is bounded by a similar quantity with a larger constant $C$. Lemma \ref{lem:lp} and the Sobolev embedding theorem implies that
\begin{equation*}
	\norm{\tilde{w}}_{C^3(A_1)}\leq C\varepsilon \left( e^{-\frac{L}{2}(l-l_0)} + e^{-\frac{L}{2} (l_i-l)} \right).	 
\end{equation*}
Noticing the fact that $\tilde{\nabla}_{S^3} u_i^*$ is always zero, we have
\begin{equation*}
	\norm{\tilde{\nabla}_{S^3} \tilde{u}}_{C^2(A_1)}\leq C\varepsilon \left( e^{-\frac{L}{2}(l-l_0)} + e^{-\frac{L}{2} (l_i-l)} \right).	
\end{equation*}
Scaling back, we have
\begin{equation*}
	\max_{A_1} \abs{\partial_r^p \tilde{\nabla}^q_{S^3} u_i} \leq C\varepsilon \left( e^{-\frac{L}{2}(l-l_0)} + e^{-\frac{L}{2} (l_i-l)} \right).
\end{equation*}
The second inequality is trivial from (\ref{eqn:tangentdecay}).
\end{proof}

\section{decay of radial energy}\label{sec:rad}

In previous section, we showed that in (\ref{eqn:tangentdecay}) that the tangential derivative of $u_i$ satisfies some decay estimate. We will show in this section that this is also true for radial derivative of $u_i$. Argument of this kind usually uses the so called Pohozaev estimate, which was first introduced into the neck analysis of harmonic maps by Lin and Wang in \cite{LW}. It has been generalized to the case of biharmonic maps by various authors, see for example \cite{HM}, \cite{LR}, \cite{WZ12}.

In this paper, we use essentially the same computation. However, instead of deriving an inequality relating the tangential energy and the radial energy, we obtain an ODE for the radial energy on the boundary of balls, in which the tangential energy appears as coefficients. Our result is proved with the help of this ODE.

It turns out the computation is easier and clearer in cylinder coordinates. Recall that in polar coordinates in $\Real^4$,
\begin{equation*}
	\triangle u= \left( \frac{\partial^2}{\partial r^2} +\frac{3}{r}\pfrac{}{r} +\frac{1}{r^2}\tilde{\triangle}_{S^3}  \right)u.
\end{equation*}
Here $\tilde{\triangle}_{S^3}$ is the Laplace operator on the standard $S^3$. By setting $r=e^{t}$, we have
\begin{equation*}
	\triangle u= e^{-2t} \left( \partial_t^2 + 3\partial_t  +\tilde{\triangle}_{S^3}  \right)u.
\end{equation*}
Direct computation shows that
\begin{eqnarray}\label{eqn:twolaplace}
	\triangle^2 u&=&  e^{-4t} \left( \partial_t^2 + \tilde{\triangle}_{S^3} -2 \partial_t \right) \left( \partial_t^2 +\tilde{\triangle}_{S^3} + 2 \partial_t \right) u \\\nonumber
	&=& e^{-4t} \left( (\partial_t^2 +\tilde{\triangle}_{S^3})^2 -4\partial_t^2 \right)u.
\end{eqnarray}

Suppose that $u$ is a biharmonic map defined on $B_r$. Recall that
$u$ is a biharmonic map if and only if $\triangle^2 u$ is normal to the tangent space $T_u N$. On the other hand $\partial_t u$ is a tangent vector at $u(x)\in N$. Therefore,
\begin{equation*}
	\int_{S^3} \triangle^2 u \cdot \partial_t u d\theta =0
\end{equation*}
for all $t$, where $d\theta$ is the volume element of $S^3$.
By (\ref{eqn:twolaplace}), we have
\begin{equation}\label{eqn:start}
	\int_{S^3} \partial_t u \partial_t ^4 u + \partial_t u \tilde{\triangle}_{S^3}^2 u + 2 \partial_t u \partial_t^2  \tilde{\triangle}_{S^3} u -4 \partial_t u \partial_t^2 u d\theta =0.
\end{equation}
By integrating by parts and noticing that
\begin{equation*}
	\partial_t u \partial_t^4 u= \partial_t \left( \partial_t u \partial_t^3 u -\frac{1}{2}\abs{\partial_t^2 u}^2 \right),
\end{equation*}
we obtain
\begin{equation}\label{eqn:step1}
	\partial_t \int_{S^3} 2 \partial_t u \partial_t^3 u - \abs{\partial_t^2 u}^2 +\abs{\tilde{\triangle}_{S^3} u}^2 -2 \abs{\partial_t \tilde{\nabla}_{S^3} u}^2 -4 \abs{\partial_t u}^2 d\theta =0.
\end{equation}

We claim that
\begin{equation*}
	\lim_{t\to -\infty}\int_{S^3} 2 \partial_t u \partial_t^3 u - \abs{\partial_t^2 u}^2 +\abs{\tilde{\triangle}_{S^3} u}^2 -2 \abs{\partial_t \tilde{\nabla}_{S^3} u}^2 -4 \abs{\partial_t u}^2 d\theta =0.
\end{equation*}
To see this, $u$ is a smooth map defined on $B_r$ for some $r>0$. The limit $t\to -\infty$ is the same as the limit $r\to 0$. It suffices to translate back the integral into polar coordinates and note that
\begin{equation*}
	\partial_r u, \partial_r^2 u, \partial_r^3 u, \frac{1}{r}\tilde{\nabla}_{S^3} u, \frac{1}{r^2} \tilde{\triangle}_{S^3} u, \frac{1}{r} \partial_r\tilde{\nabla}_{S^3} u
\end{equation*}
are bounded near the origin.

Integrating (\ref{eqn:step1}) from $-\infty$ to $t$, we get
\begin{equation*}
	\int_{S^3} 2 \partial_t u \partial_t^3 u - \abs{\partial_t^2 u}^2 +\abs{\tilde{\triangle}_{S^3} u}^2 -2 \abs{\partial_t \tilde{\nabla}_{S^3} u}^2 -4 \abs{\partial_t u}^2 d\theta =0
\end{equation*}
for all $t$.
Using
\begin{equation*}
	\partial_t u \partial_t^3 u =\partial_t (\partial_t u \partial_t^2 u)-\abs{\partial_t^2 u}^2,
\end{equation*}
the above equation can be written as
\begin{equation}\label{eqn:ode}
	\partial_t \int_{S^3}  \partial_t u \partial_t ^2 u d\theta -\int_{S^3} \frac{3}{2} \abs{\partial_t^2 u}^2 +2 \abs{\partial_t u}^2 d\theta =\Theta(t),
\end{equation}
where
\begin{equation*}
	\Theta(t)=\int_{S^3}-\frac{1}{2}\abs{\tilde{\triangle}_{S^3} u}^2 + \abs{\partial_t \tilde{\nabla}_{S^3} u}^2.
\end{equation*}
 This is the ODE that we mentioned at the beginning of this section.

Now, we apply the above computation to the sequence of biharmonic maps $u_i$.  $u_i$ as a function of $(t,\theta)$ satisfies (\ref{eqn:ode}). By (\ref{eqn:weuse}), we know that for $t\in [\log \lambda_i R, \log \delta]$,
\begin{equation*}
	\abs{\Theta_i(t)}\leq C\varepsilon^2 \left( e^{- (\log \delta -t)}+e^{-(t-\log (\lambda_i R))} \right).
\end{equation*}

Moreover, by $\varepsilon_0-$regularity (Theorem \ref{thm:regularity}) and (\ref{eqn:DT}), we have
\begin{equation*}
	\max_{t\in [\log \lambda_i R, \log \delta]}\max_{\theta\in S^3} \abs{\tilde{\nabla}^k u}\leq C \varepsilon.
\end{equation*}
for $k\leq3$.
Here $\tilde{\nabla}$ is the gradient of $[\log (\lambda_i  R), \log \delta]\times S^3$ with the product metric.
Hence, by integrating (\ref{eqn:ode}) from $[\log \lambda_i R, \log \delta]$, we have
\begin{equation}\label{eqn:total}
	\int_{\log \lambda_i R}^{\log \delta} \int_{S^3}\frac{3}{2} \abs{\partial_t^2 u}^2 + \abs{\partial_t u}^2 d\theta dt \leq C\varepsilon^2.
\end{equation}
\begin{rem}
	We remark that in fact, the argument above gives an independent proof of the energy identity in the blow up analysis of biharmonic maps.
\end{rem}

For some fixed $t_0\in [\log \lambda_i R, \log \delta]$, set
\begin{equation*}
	F(t)=\int_{t_0-t}^{t_0+t}\int_{S^3} \frac{3}{2} \abs{\partial_t^2 u}^2 + 2\abs{\partial_t u}^2 d\theta dt.
\end{equation*}
$F$ is defined for $0\leq t\leq \min \set{t_0- \log \lambda_i R, \log \delta -t_0}$. Integrating (\ref{eqn:ode}) from $t_0-t$ to $t_0+t$, we obtain
\begin{eqnarray*}
	F(t)&\leq& \frac{1}{2\sqrt{3}} \left( \int_{\set{t_0-t}\times S^3} +\int_{\set{t_0+t}\times S^3} \right) \frac{3}{2}\abs{\partial_t^2 u}^2 + 2 \abs{\partial_t u}^2 d\theta\\
	&& + \int_{t_0-t}^{t_0+t} \abs{\Theta_i(s)} ds.
\end{eqnarray*}
Direct computation shows
\begin{equation*}
	\int_{t_0-t}^{t_0+t} \abs{\Theta_i(s)} ds \leq C\varepsilon^2 \left( e^{- (\log \delta -t_0)}+ e^{- (t_0-\log \lambda_i R)} \right) e^{t}.
\end{equation*}
Hence,
\begin{eqnarray*}
	F(t)&\leq& \frac{1}{2}\partial_t F(t)+C\varepsilon^2 \left( e^{- (\log \delta -t_0)}+ e^{- (t_0-\log \lambda_i R)} \right) e^{ t}.
\end{eqnarray*}
Multiplying $e^{-2t}$ to both sides of the inequality, we have
\begin{equation*}
	(e^{-2t} F(t))'\geq -C\varepsilon^2 \left( e^{- (\log \delta -t_0)}+ e^{- (t_0-\log \lambda_i R)} \right) e^{-t}.
\end{equation*}
We assume without loss of generality that $\log \delta- t_0\leq t_0-\log \lambda_i R$. Then, we integrate the above inequality from $t=1$ to $t=\log \delta -t_0$ to get
\begin{eqnarray*}
	F(1)&\leq& e^{-2(\log \delta -t_0)+2} F(\log \delta- t_0) + C\varepsilon^2 \left( e^{- (\log \delta -t_0)}+ e^{- (t_0-\log \lambda_i R)} \right) \\
	&\leq&  C\varepsilon^2 \left( e^{- (\log \delta -t_0)}+ e^{- (t_0-\log \lambda_i R)} \right).
\end{eqnarray*}
Here we used (\ref{eqn:total}).

By the Lemma \ref{lem:point}, we have
\begin{equation*}
\int_{t_0-1}^{t_0+1}\int_{S^3}|\tilde{\nabla}^2u|^2+|\tilde{\nabla}u|^2d\theta dt\leq C\varepsilon^2 \left( e^{- (\log \delta -t_0)}+ e^{- (t_0-\log \lambda_i R)} \right).
\end{equation*}

Direct computation shows that
\begin{eqnarray*}
\int_{B_{e^{t_0+1}}\setminus B_{e^{t_0-1}}}|\nabla^2u|^2+\frac{1}{|x|^2}|\nabla u|^2dx
&\leq&C\int_{t_0-1}^{t_0+1}\int_{S^3}|\tilde{\nabla}^2u|^2+|\tilde{\nabla}u|^2d\theta dt \\
&\leq&C\varepsilon^2 \left( e^{- (\log \delta -t_0)}+ e^{- (t_0-\log \lambda_i R)} \right).
\end{eqnarray*}

Then by Sobolev embedding and the $\varepsilon-$regularity (Theorem \ref{thm:regularity}), we have
\begin{equation*}
	\max_{|x|=e^{t_0}}|x|^k\abs{\nabla^k u}\leq C\varepsilon \left( e^{- \frac{1}{2}(\log \delta -t_0)}+ e^{- \frac{1}{2}(t_0-\log \lambda_i R)} \right)
\end{equation*}
for $k\leq 3$.

So, we proved the decay of first derivative of $u(x)$. It is easy to derive the no neck estimate from here. Hence, we complete the proof of Theorem \ref{thm:main}.

\section{another proof of the removable singularity}\label{sec:rem}

In this section, we will give a proof of the removable singularity for biharmonic maps in dimension 4 following the argument of Sacks and Uhlenbeck in \cite{SU}. Precisely, we prove
\begin{thm}\label{l1}
	Suppose $u\in C^\infty(B_1\setminus \set{0})$ is a biharmonic map and satisfies
	\begin{equation*}
		\int_{B_1} \abs{\nabla^2 u}^2 + \abs{\nabla u}^4 dx \leq C<+\infty,
	\end{equation*}
then $u\in C^\infty(B_1)$.
\end{thm}

In previous sections, to prove Theorem \ref{thm:main}, we study the behavior of biharmonic maps on the neck region $B_\delta\setminus B_{\lambda_i R}$ and proved that in terms of cylinder coordinates $(t,\theta)$, the derivatives of $u$ (with respect to the cylinder coordinates) decay with the distance to both ends of the cylinder $[\log \lambda_i R, \log \delta]\times S^3$. It is natural to expect that the argument can be applied to the study of isolated singularities.

For any $\varepsilon>0$, by shrinking the size of the ball, we assume without loss of generality that
	\begin{equation*}
		\int_{B_1} \abs{\nabla^2 u}^2 + \abs{\nabla u}^4 dx \leq \varepsilon^4.
	\end{equation*}
By Theorem \ref{thm:regularity}, we have
\begin{equation}\label{eqn:kaka}
	\abs{\nabla^k u}\leq \frac{\varepsilon}{\abs{x}^k},
\end{equation}
where $k\leq3$.

Set as before
\begin{equation*}
	u^*(r)=\frac{1}{\abs{\partial B_r}} \int_{\partial B_r} u d\theta.	
\end{equation*}
Lemma \ref{lem:beapprox} implies that $w=u-u^*$ is an $\eta_0-$approximate biharmonic function if $\varepsilon$ is chosen to be small. Set
\begin{equation*}
	F_l=\int_{A_l} \frac{1}{\abs{x}^4}\abs{w}^2 dx.
\end{equation*}
Similar to the proof of Theorem \ref{thm:tcapprox}, we claim that
\begin{equation*}
	F_l\geq e^L F_{l+1}
\end{equation*}
for any $l>2$.
If this is not true, by (c) of Theorem \ref{thm:tcapprox}, there is some $l_0>2$ such that
\begin{equation*}
	F_{l_0+1}\leq e^{-L} F_{l_0+2}
\end{equation*}
and by (b) of the same theorem, we know that for all $l>l_0+1$
\begin{equation*}
	F_{l}\leq e^{-L} F_{l+1}.
\end{equation*}
However, this is not possible since
\begin{equation*}
	F_l \leq C\int_{[{-lL},{-(l-1)L}]\times S^3} \abs{w}^2 d\theta dt \leq  C\int_{[{-lL},{-(l-1)L}]\times S^3} \abs{\nabla u}^2 d\theta dt
\end{equation*}
and $u$ as a function of $(t,\theta)$ has bounded energy on the cylinder $(-\infty,0]\times S^3$ (see (\ref{eqn:kaka})).
The same argument as before we know that for any $p+q\leq3$ and $q\geq 1$ and $t\in (-\infty,0]$
\begin{equation}\label{eqn:thetadecay}
	\abs{\partial_t^p \tilde{\nabla}_{S^3}^q u} (t,\theta)\leq C\varepsilon e^{\frac{t}{2}} .
\end{equation}
The proof of Section \ref{sec:rad} implies
\begin{equation}\label{eqn:ode2}
	\partial_t \int_{S^3}  \partial_t u \partial_t ^2 u d\theta -\int_{S^3} \frac{3}{2} \abs{\partial_t^2 u}^2 +2 \abs{\partial_t u}^2 d\theta =\Theta(t)
\end{equation}
defined for $u$ as a function $(t,\theta)$ with $\abs{\Theta}\leq C\varepsilon^2 e^{t}$.
\begin{rem}\label{rem:zero}
	In the proof of (\ref{eqn:ode2}) in Section \ref{sec:rad}, we need to justify that the limit of
	\begin{equation}\label{eqn:integrand}
		\int_{S^3} 2 \partial_t u \partial_t^3 u - \abs{\partial_t^2 u}^2 +\abs{\tilde{\triangle}_{S^3} u}^2 -2 \abs{\partial_t \tilde{\nabla}_{S^3} u}^2 -4 \abs{\partial_t u}^2 d\theta
	\end{equation}
	is zero when $t\to -\infty$.
However, $u$ is not smooth at $0$ as in Section \ref{sec:rad}. Fortunately, we have
\begin{equation*}
	\int_{B_1} \abs{\nabla^2 u}^2 + \abs{\nabla u}^4 dx <+\infty.
\end{equation*}
Therefore, Theorem \ref{thm:regularity} implies
\begin{equation*}
	\max_{B_\rho} {\rho}^k \abs{\nabla^k u} = o(1)
\end{equation*}
as $\rho\to 0$. It follows that the integrand of (\ref{eqn:integrand}) goes to zero when $t\to -\infty$.
\end{rem}

For any $t_0<-1$, we define for $t\in (0,-t_0-1)$
\begin{equation*}
	F(t)=\int_{t_0-t}^{t_0+t} \int_{S^3} \frac{3}{2} \abs{\partial_t^2 u}^2 +2 \abs{\partial_t u}^2 d\theta ds
\end{equation*}
Integrating (\ref{eqn:ode2}), we have
\begin{equation*}
	F(t)\leq \left( \int_{\set{t_0-t}\times S^3} +\int_{\set{t_0+t}\times S^3} \right) \abs{\partial_t u} \abs{\partial_t^2 u} d\theta +  C\varepsilon^2 \int_{t_0-t}^{t_0+t} e^s ds.
\end{equation*}
Hence,
\begin{equation*}
	F(t)\leq \frac{1}{2} F'(t) + C\varepsilon^2  e^{t_0+t}.
\end{equation*}
Multiplying $-e^{-2t}$ to both sides of the above inequality and integrating from $t=1$ to $t=-t_0-1$, we get
\begin{equation*}
	e^{2 t_0} F(-t_0-1) - e^{-2} F(1)\geq -C\varepsilon^2 e^{t_0}.
\end{equation*}
Therefore,
\begin{equation}\label{eqn:tdecay}
	F(1)\leq C \varepsilon^2 e^{t_0} + C e^{2 t_0} F(-t_0-1).
\end{equation}
We claim that $F(-t_0-1)$ is uniformly bounded by $C\varepsilon ^2$ with respect to $t_0$. This follows from the fact that
\begin{equation*}
	\int_{-\infty}^{-1} \int_{S^3} \frac{3}{2} \abs{\partial_t^2 u}^2 +2 \abs{\partial_t u}^2 d\theta dt <C\varepsilon^2 .
\end{equation*}
To see this, we integrate (\ref{eqn:ode2}) from $-\infty$ to $0$. It suffices to show that
\begin{equation*}
	\lim_{t\to -\infty} \int_{S^3} \abs{\partial_t u} \abs{\partial_t^2 u} d\theta =0.
\end{equation*}
The reason is the same as in Remark \ref{rem:zero}.

In summary, we have shown that
\begin{equation*}
\int_{t_0-1}^{t_0+1}\int_{S^3}|\tilde{\nabla}^2u|^2+|\tilde{\nabla}u|^2d\theta dt\leq C\varepsilon^2 e^{t_0}.
\end{equation*}

Then the same arguments in the previous section tells us
\begin{equation*}
	|x|\abs{\nabla u}\leq C\varepsilon |x|^{\frac{1}{2}}.
\end{equation*}

This concludes that $u$ is H\"older continuous as a function defined on $B_1$. Higher regularity follows from the proof of Theorem 5.1 of \cite{CWY}.

\section{intrinsic biharmonic maps}\label{sec:intrinsic}
When we consider intrinsic biharmonic maps, the key difference in the proof is the Pohozaev type argument in Section \ref{sec:rad}. Here we shall show that why the proof is robust enough so that it works for intrinsic biharmonic maps as well.

The starting point of the argument in Section \ref{sec:rad} is (\ref{eqn:start}), which uses the fact that $u$ is extrinsic biharmonic map if and only if $\triangle^2 u$ is normal to the tangent space $T_u N$. For intrinsic biharmonic maps, this is no longer true. However, we recognize that the right hand side of (\ref{eqn:start}) is just
\begin{equation}\label{eqn:rhs}
	r\int_{\partial B_r} P(u) (\triangle^2 u) \cdot r\partial_r u d\sigma.
\end{equation}
Here $P(u)$ is the projection to $T_uN$. It is known that the Euler Lagrange equation of intrinsic biharmonic maps are of the form
\begin{equation*}
	P(u)\left( \triangle^2 u + \mbox{additional terms} \right)=0.
\end{equation*}

Next, we show case by case how to modify the argument in Section \ref{sec:rad}.

\subsection{intrinsic Laplace biharmoinc maps}
Since
\begin{equation*}
	\int \abs{\tau(u)}^2 dx = \int \abs{\triangle u}^2 - \abs{B(u)(\nabla u,\nabla u)}^2 dx,
\end{equation*}
the additional term is contributed by the variation of $\int \abs{B(u)(\nabla u,\nabla u)}^2 dx$.

Let's consider the variation given by $u_t=\Pi (u +t \varphi)$. Here $\Pi$ is the nearest point projection to $N$ defined in a neighborhood of $N$. Compute
\begin{eqnarray*}
	\frac{d}{dt}|_{t=0} \int \abs{B(\nabla u_t,\nabla u_t)}^2 &=& 2\int B(\nabla u,\nabla u) \nabla_u B (\nabla u,\nabla u) P(u) \varphi \\
	&& + 2 B(\nabla u,\nabla u) B(\nabla u, \nabla (P(u) \varphi) )
\end{eqnarray*}
The contribution to the Euler-Lagrange equation of this part is
\begin{equation}\label{eqn:addition1}
	I=P(u)\left[ 2B(\nabla u,\nabla u) \nabla_u B (\nabla u,\nabla u) - 4 \sum_i \nabla_i\left( B(\nabla u,\nabla u) B(\nabla_i u, \cdot) \right) \right].
\end{equation}
For a better understanding of the above terms, we use local coordinates. Let $x^i$ be coordinate system of $\Omega$ and $y^\alpha$ be the coordinates of $\Real^K$ in which $N$ is embedded. We extend the domain of $B$ to a neighborhood of $N$. Hence,
\begin{equation*}
	B^\alpha (\nabla u,\nabla u) =  B^{\alpha}_{\beta\gamma} \partial_i u^\beta \partial_i u^\gamma.
\end{equation*}
(\ref{eqn:addition1}) in coordinates is
\begin{equation*}
	P(u)^\alpha_\beta \left[ 2B^\gamma(\nabla u,\nabla u) \partial_{y_\alpha} B^\gamma (\nabla u,\nabla u) - 4 \sum_i \partial_{x_i} \left( B^\gamma (\nabla u,\nabla u) B^\gamma_{\eta\alpha} \partial _i u^\eta \right) \right].
\end{equation*}
To follow the computation in Section \ref{sec:rad}, we shall multiply the Euler-Lagrange equation by $x^k \partial_k u^\beta$ and integrate over $\partial B_r$. Before that, we need the following lemma.

\begin{lem}\label{lem:div}
	Let $X=(X_1,\cdots,X_4)$ be a vector field. We have
	\begin{equation*}
		\mbox{div} X  = \partial_r (X_r) +\frac{3}{r}X_r +\mbox{div}_{S^3} X^T.
	\end{equation*}
	Here $r=\abs{x}$ and $X_r=(X,\partial_r)$ and $X^T$ is the projection of $X$ to the tangent space of $\partial B_r$, $\mbox{div}_{S^3}$ is the divergence operator of $\partial B_r$.
\end{lem}

\begin{proof}
	The proof is basic computation. We present it for the sake of completeness. Let $\set{\omega_i}$ be an orthonormal frame of the tangent bundle of $\partial B_r$(locally). Due to the decomposition $X=X_r \partial_r +X^T$, we have
\begin{eqnarray*}
	\mbox{div} X &=& (\nabla_{\partial_r} X,\partial_r) +\sum_i (\nabla_{\omega_i} X, \omega_i) \\
	&=& ( \nabla_{\partial_r} (X_r \partial_r) ,\partial_r) + (\nabla_{\partial_r} X^T ,\partial_r) \\
	&& + \sum_{i} (\nabla_{\omega_i} (X_r \partial_r), \omega_i) + (\nabla _{\omega_i} X^T, \omega_i) \\
	&=& \partial_r(X_r) + \frac{3}{r} X_r + \mbox{div}_{S^3} X^T.
\end{eqnarray*}
Here we have used the following facts from Riemannian geometry:

(1) $\nabla_{\partial_r}\partial_r=0$;

(2) $(\nabla_{\partial_r} X^T,\partial_r)=-(X^T,\nabla_{\partial_r}\partial_r)=0$;

(3) $(\nabla_{\omega_i}\partial_r,\omega_j)=\frac{1}{r}\delta_{ij}$.
\end{proof}

Now we may proceed to compute the effect of the additional term $I$ on the Pohozaev inequality. For simplicity, we split $I$ into $I_1-I_2$ (as is obvious in (\ref{eqn:addition1})) and compute
\begin{equation*}
	\int_{\partial B_r} I_2 r\partial_r u d\sigma.
\end{equation*}
Since $r\partial_r u$ is a tangent vector of $T_u N$, we may forget the $P(u)$ in $I$. We notice that the remaining part of $I_2$ is the divergence of
\begin{equation*}
	X=4 (B^\gamma(\nabla u,\nabla u) B^\gamma_{\eta \alpha}\partial_i u^\eta).
\end{equation*}
Hence, we may apply the above lemma to get

\begin{eqnarray*}
	\int_{\partial B_r} I_2 r\partial_r u d\sigma &=&  \int_{\partial B_r} \left( \partial_r X_r + \frac{3}{r} X_r + \mbox{div}_{S^3} X^T \right) r\partial_r u d\sigma \\
	&=&  \int_{\partial B_r} \partial_r \left( 4B(\nabla u,\nabla u)B_{\eta\alpha}\partial_r u^\eta \right) r\partial_r u^\alpha \\
	&& + 12 B(\nabla u,\nabla u) B(\partial_r u,\partial_r u) \\
	&& - 4 B(\nabla u,\nabla u) B(\nabla_{S^3} u, \nabla_{S^3}(r\partial_r u)) d\sigma,
\end{eqnarray*}
where $\nabla_{S^3}$ the gradient of $\partial B_r$.

Using cylinder coordinates $(t,\theta)$ where $r=e^t$, the first line above becomes
\begin{eqnarray*}
	&&  \int_{S^3} r^3 \partial_r \left( 4B(\nabla u,\nabla u)B_{\eta\alpha}\partial_r u^\eta \right) r\partial_r u^\alpha d\theta \\
	&=&  \int_{S^3} r^3 \partial_r \left[ \frac{1}{r^3} \left( 4B(\tilde{\nabla} u,\tilde{\nabla} u) B_{\eta\alpha} \partial_t u^\eta \right) \right] \partial_t u^\alpha d\theta \\
	&=& \frac{1}{r} \int_{S^3} -12  B(\tilde{\nabla} u, \tilde{\nabla} u) B(\partial_t u,\partial_t u)  + \partial_t \left[  \left( 4B(\tilde{\nabla} u,\tilde{\nabla} u) B_{\eta\alpha} \partial_t u^\eta \right) \right] \partial_t u^\alpha d\theta.
\end{eqnarray*}
Here $\tilde{\nabla}$ is the gradient on $S^3\times (-\infty,0]$ with product metric.

 Notice that the second line cancels with the first term above. Hence, we have
\begin{eqnarray*}
	&& \int_{\partial B_r} I_2 r\partial_r u d\sigma\\
	&=& \frac{1}{r}\int_{S^3} 4 \partial_t (B(\tilde{\nabla} u,\tilde{\nabla} u) B(\partial_t u,\partial_t u)) + 2 B(\tilde{\nabla}u,\tilde{\nabla}u) \nabla_u B(\tilde{\nabla}u,\tilde{\nabla}u) \partial_t u  \\
	&& - 2 B(\tilde{\nabla }u, \tilde{\nabla}u ) \partial_t (B(\partial_t u, \partial_t u)) - 2 B(\tilde{\nabla}u, \tilde{\nabla}u)\partial_t (B(\tilde{\nabla}_{S^3} u, \tilde{\nabla}_{S^3} u)) d\theta.
\end{eqnarray*}

In summary,
\begin{eqnarray*}
	&& \int_{\partial B_r} (I_1-I_2) r\partial_r u d\sigma \\
	&=& \frac{1}{r} \partial_t \int_{S^3} -4 B(\tilde{\nabla} u,\tilde{\nabla} u) B(\partial_t u,\partial_t u) + \abs{B(\tilde{\nabla}u, \tilde{\nabla}u)}^2 d\theta.
\end{eqnarray*}

If we multiply the Euler-Lagrange equation of intrinsic Laplace biharmonic map with $r\partial_r u$ and integrate over $\partial B_r$, we obtain
\begin{eqnarray}\label{eqn:start2}
	&& \int_{S^3} \partial_t u \partial_t ^4 u + \partial_t u \tilde{\triangle}_{S^3}^2 u + 2 \partial_t u \partial_t^2  \tilde{\triangle}_{S^3} u -4 \partial_t u \partial_t^2 u d\theta \\\nonumber
	&& + \partial_t \int_{S^3} 2 B(\tilde{\nabla}u, \tilde{\nabla} u) B(\partial_t u,\partial_t u) -\frac{1}{2} \abs{B(\tilde{\nabla} u, \tilde{\nabla} u)}^2 d\theta =0.
\end{eqnarray}

Rewriting the first line as before and integrating over $(-\infty,t)$ again, we obtain
\begin{eqnarray*}
	0&=& \int_{S^3}  \partial_t u \partial_t^3 u -\frac{1}{2} \abs{\partial_t^2 u}^2 +\frac{1}{2}\abs{\tilde{\triangle}_{S^3} u}^2 - \abs{\partial_t \tilde{\nabla}_{S^3} u}^2 -2 \abs{\partial_t u}^2 \\
	&&+2 B(\tilde{\nabla}u, \tilde{\nabla} u) B(\partial_t u,\partial_t u) -\frac{1}{2}\abs{B(\tilde{\nabla} u, \tilde{\nabla} u)}^2 d\theta.
\end{eqnarray*}
The $\partial_t u \partial_t^3 u$ term is dealt with as before and we move everything involving tangential derivative to the right to get
\begin{equation*}
	\partial_t \int_{S^3}  \partial_t u \partial_t ^2 u d\theta -\int_{S^3} \frac{3}{2} \abs{\partial_t^2 u}^2 +2 \abs{\partial_t u}^2  - \Psi(t) d\theta =\Theta(t),
\end{equation*}
where
\begin{eqnarray*}
	\Theta(t)=\int_{S^3}-\frac{1}{2}\abs{\tilde{\triangle}_{S^3} u}^2 + \abs{\partial_t \tilde{\nabla}_{S^3} u}^2
	-B(\partial_t u,\partial_t u)B(\tilde{\nabla}_{S^3} u,\tilde{\nabla}_{S^3} u)
	+\frac{1}{2}\abs{B(\tilde{\nabla}_{S^3} u,\tilde{\nabla}_{S^3} u)}^2
\end{eqnarray*}
 and
\begin{equation*}
  \Psi=\frac{3}{2}\abs{B(\partial_t u,\partial_t u)}^2.
  \end{equation*}

 Noticing that $\Psi$ is a fourth order polynomial of $\partial_t u$. By $\varepsilon_0-$regularity (Theorem \ref{thm:regularity}), if $\varepsilon$ in (\ref{eqn:DT}) is chosen to be small, we have
\begin{equation*}
	\abs{\Psi(t)}\leq \frac{1}{2} \abs{\partial_t u}^2.
\end{equation*}
\begin{rem}
	This is exactly why the additional term causes no trouble. The contribution to the radial part is a fourth order term. By $\varepsilon_0-$regularity, it is controlled by a second order term with a small coefficient and can be absorbed into the positive term.
\end{rem}

Therefore, by setting
\begin{equation*}
	F(t)=\int_{t_0-t}^{t_0+t}\int_{S^3} \frac{3}{2} \abs{\partial_t^2 u}^2 +\frac{3}{2}\abs{\partial_t u}^2 d\theta dt,
\end{equation*}
we have
\begin{eqnarray*}
	F(t)&\leq& \left( \int_{\set{t_0-t}\times S^3}+\int_{\set{t_0+t}\times S^3} \right) \abs{\partial_t u}\abs{\partial_t^2 u} d\theta +  \int_{t_0-t}^{t_0+t} \abs{\Theta(t)} dt \\
	&\leq& \frac{1}{2} F'(t) + \int_{t_0-t}^{t_0+t} \abs{\Theta(t)} dt.
\end{eqnarray*}

The rest of the proof is the same as the case of extrinsic biharmonic map.

\subsection{intrinsic Hessian biharmonic map}
We are interested in the Euler-Lagrange equation of the intrinsic Hessian biharmonic map. As noted in \cite{Moser}, it is the same as the Euler-Lagrange equation of the functional
\begin{equation*}
	\int_\Omega \abs{\tau(u)}^2 + \langle R(u)\left( \partial_i u, \partial_j u \right) \partial_j u, \partial_i u \rangle dx.
\end{equation*}
We want to compute the effect of this additional curvature term on the Pohozaev argument and show that the previous proof works for this case as well.

If the variation is given by $u_t=\Pi(u+t\varphi)$, then the variation of the additional term is
\begin{eqnarray*}
	&& \frac{d}{dt}|_{t=0} \int \langle R(u_t)\left( \partial_i u_t, \partial_j u_t \right) \partial_j u_t, \partial_i u_t \rangle dx \\
	&=& \int (\nabla_u R) (\nabla u,\cdots) P(u)\varphi \\
	&&+ \sum_{(\alpha)} R_{\alpha\beta\gamma\delta} \partial_i (P^\alpha(u)\varphi) \partial_j u^\beta \partial_j u^\gamma \partial_i u^\delta dx.
\end{eqnarray*}
Here by $\sum_{(\alpha)}$, we mean a summation of four terms and the other three are similar and can be obtained by replacing $\alpha$ with $\beta,\gamma$ or $\delta$.

Hence, in comparison with the Euler-Lagrange equation of the intrinsic Laplace biharmonic maps, there is an additional term
\begin{equation*}
	J:=J_1-J_2=P(u)\left[ (\nabla_u R)(\nabla u,\cdots) -\sum_{(\alpha)} \partial_i (R_{\alpha\beta\gamma\delta}\partial_j u^\beta \partial_j u^\gamma \partial_i u^\delta) \right].
\end{equation*}

Applying Lemma \ref{lem:div}, we have
\begin{eqnarray*}
	\int_{\partial B_r} J_2 r \partial_r u d\sigma &=& \sum_{(\alpha)} \int_{\partial B_r} \partial_r \left( R_{\alpha\beta\gamma\delta}\partial_j u^\beta \partial_j u^\gamma \partial_r u^\delta \right) r \partial_r u^\alpha \\
	&&+  3 R_{\alpha\beta\gamma\delta} \partial_r u^\alpha \partial_j u^\beta \partial_j u^\gamma \partial_r u^\delta \\
	&& -  R_{\alpha\beta\gamma\delta} \nabla_{S^3} (r\partial_r u^\alpha) \partial_j u^\beta \partial_j u^\gamma \nabla_{S^3} u^\delta\, d\sigma \\
	&=& \sum_{(\alpha)} \int_{S^3} r^3 \partial_r \left( \frac{1}{r^3} R_{\alpha\beta\gamma\delta} \tilde{\nabla}_i u^\beta \tilde{\nabla}_{i} u^\gamma \partial_t u^\delta \right) r \partial_r u^\alpha \\
	&& + \frac{3}{r} R_{\alpha\beta\gamma\delta} \partial_t u^\alpha \tilde{\nabla}_i u^\beta \tilde{\nabla}_i u^\gamma \partial_t u^\delta \\
	&& -\frac{1}{r} R_{\alpha\beta\gamma\delta} \tilde{\nabla}_{S^3} (\partial_t u^\alpha) \tilde{\nabla}_i u^\beta \tilde{\nabla}_i u^\gamma \tilde{\nabla}_{S^3} u^\delta d\theta \\
	&=& \frac{1}{r} \sum_{(\alpha)} \int_{S^3} \partial_t \left( R_{\alpha\beta\gamma\delta} \partial_t u^\alpha \tilde{\nabla}_i u^\beta \tilde{\nabla}_i u^\gamma \partial_t u^\delta \right) \\
	&& -R_{\alpha\beta\gamma\delta} \partial^2_t u^\alpha \tilde{\nabla}_i u^\beta \tilde{\nabla}_i u^\gamma \partial_{t} u^\delta- R_{\alpha\beta\gamma\delta} \tilde{\nabla}_{S^3} (\partial_t u^\alpha) \tilde{\nabla}_i u^\beta \tilde{\nabla}_i u^\gamma \tilde{\nabla}_{S^3} u^\delta d\theta
\end{eqnarray*}
The symmetry of Riemann curvature tensor implies that
\begin{eqnarray*}
	\int_{\partial B_r} (J_1-J_2) r \partial_r u d\sigma 	&=& \frac{1}{r} \int_{S^3} - \partial_t \left(  4 R_{\alpha\beta\gamma\delta} \partial_t u^\alpha \tilde{\nabla}_i u^\beta \tilde{\nabla}_i u^\gamma \partial_t u^\delta \right) \\
	&& + \partial_t \left( R_{\alpha\beta\gamma\delta} \tilde{\nabla}_j u^\alpha \tilde{\nabla}_i u^\beta \tilde{\nabla}_i u^\gamma \tilde{\nabla}_j u^\delta \right) d\theta
\end{eqnarray*}

The rest of the proof are the same as in the previous subsection.

\section*{Appendix}
\begin{proof}[Proof of Theorem \ref{thm:regularity}]
It will be convenient to assume that $\overline{u}=0$. Since $u$ is a biharmonic map, then it satisfies the Euler-Lagrange
\begin{eqnarray*}
\triangle^2u=\nabla^3u\#\nabla u+\nabla^2 u\#\nabla^2 u+\nabla^2 u\#\nabla u\#\nabla u+\nabla u\#\nabla u\#\nabla u\#\nabla u.
\end{eqnarray*}
Let $0<\sigma<1$ and $\sigma'=\frac{1+\sigma}{2}$, take cut-off function $\varphi\in C_0^\infty(B_{\sigma'})$ satisfying $\varphi\equiv1$ in $B_\sigma$, $|\nabla\varphi|\leq\frac{4}{1-\sigma}$.

Direct computation shows that
\begin{eqnarray*}
\triangle^2(\varphi u)&=&\triangle (\varphi\triangle u+2\nabla u \nabla\varphi+u\triangle \varphi)\\
&=&
\varphi\triangle ^2u+4\nabla\triangle u\nabla\varphi+2\triangle u\triangle \varphi+4\nabla^2u\nabla^2\varphi+4\nabla u\nabla\triangle \varphi+u\triangle ^2\varphi\\
&=&
(\nabla^3u\#\nabla u+\nabla^2 u\#\nabla^2 u+\nabla^2 u\#\nabla u\#\nabla u+\nabla u\#\nabla u\#\nabla u\#\nabla u)\varphi\\
&&+\nabla^3 u\#\nabla\varphi+\nabla^2u\#\nabla^2\varphi+\nabla u\#\nabla^3\varphi+u\nabla^4\varphi\\
&=&
(\nabla^3(\varphi u)\#\nabla u+\nabla^2 (\varphi u)\#\nabla^2 u+\nabla^2 u\#\nabla u\#\nabla (\varphi u)+\nabla u\#\nabla u\#\nabla u\#\nabla (\varphi u))\\
&&+\nabla^3 u\#\nabla\varphi+\nabla^2u\#\nabla^2\varphi+\nabla u\#\nabla^3\varphi+u\nabla^4\varphi
+\nabla^2u\#\nabla u\#\nabla\varphi+\nabla^2\varphi\#\nabla u\#\nabla u\\
&&+\nabla u\#\nabla u\#\nabla u\#\nabla\varphi.
\end{eqnarray*}

Assume first that $1<p<\frac{4}{3}$. By the standard $L^p$ theory, we have
\begin{eqnarray*}
&&\|\nabla^4(\varphi u)\|_{L^p(B_1)}\\
&\leq& C\big(\|\nabla u\|_{L^4(B_1)}\|\nabla^3(\varphi u)\|_{L^{\frac{4p}{4-p}}(B_1)}
+\|\nabla^2 u\|_{L^2(B_1)}\|\nabla^2(\varphi u)\|_{L^{\frac{4p}{4-2p}}(B_1)}\\
&&+\|\nabla^2 u\|_{L^2(B_1)}\|\nabla u\|_{L^4(B_1)}\|\nabla(\varphi u)\|_{L^{\frac{4p}{4-3p}}(B_1)}+\|\nabla u\|^3_{L^4(B_1)}\|\nabla(\varphi u)\|_{L^{\frac{4p}{4-3p}}(B_1)} \\
&&+\frac{\|\nabla^3 u\|_{L^p(B_{\sigma'})}}{1-\sigma}+\frac{\|\nabla^2 u\|_{L^p(B_{\sigma'})}}{(1-\sigma)^2}
+\frac{\|\nabla u\|_{L^p(B_{\sigma'})}}{(1-\sigma)^3}\\
&&+\frac{\|u\|_{L^p(B_{\sigma'})}}{(1-\sigma)^4}
+\frac{\|\nabla^2 u\#\nabla u\|_{L^p(B_{\sigma'})}}{1-\sigma}+\frac{\|\nabla u\#\nabla u\|_{L^p(B_{\sigma'})}}{(1-\sigma)^2}\\
&&+\frac{1}{1-\sigma}\|\nabla u\#\nabla u\#\nabla u\|_{L^p(B_{\sigma'})} \big),
\end{eqnarray*}
By the Sobolev embedding, if $\epsilon_0$ is sufficiently small, we get
\begin{eqnarray*}
\|\nabla^4(\varphi u)\|_{L^p(B_1)}&\leq& C\big(
\frac{1}{1-\sigma}\|\nabla^3 u\|_{L^p(B_{\sigma'})}+\frac{1}{(1-\sigma)^2}\|\nabla^2 u\|_{L^p(B_{\sigma'})}
+\frac{1}{(1-\sigma)^3}\|\nabla u\|_{L^p(B_{\sigma'})}\\
&&+\frac{1}{(1-\sigma)^4}\|u\|_{L^p(B_{\sigma'})}
+\frac{1}{1-\sigma}\|\nabla^2 u\#\nabla u\|_{L^p(B_{\sigma'})}+\frac{1}{(1-\sigma)^2}\|\nabla u\#\nabla u\|_{L^p(B_{\sigma'})}\\
&&+\frac{1}{1-\sigma}\|\nabla u\#\nabla u\#\nabla u\|_{L^p(B_{\sigma'})} \big).
\end{eqnarray*}

Setting
\[
\Psi_j=\sup_{0\leq\sigma\leq1}(1-\sigma)^j\|\nabla^j u\|_{L^p(B_{\sigma})}
\]
and noticing that $1-\sigma=2(1-\sigma')$,$1<p<\frac{4}{3}$, we have
\begin{eqnarray*}
\Psi_4&\leq& C\left(\Psi_3+\Psi_2+\Psi_1+\Psi_0+\|\nabla^2 u\#\nabla u\|_{L^p(B_1)}+\|\nabla u\#\nabla u\|_{L^p(B_1)}+\|\nabla u\#\nabla u\#\nabla u\|_{L^p(B_1)}\right)\\
&\leq&
C\left(\Psi_3+\Psi_2+\Psi_1+\Psi_0+
\|\nabla^2u\|_{L^2(B_1)}+\|\nabla u\|_{L^4(B_1)}\right).
\end{eqnarray*}

Using the interpolation inequality as in Section \ref{subsec:approx}, we get
\begin{eqnarray*}
\Psi_4&\leq&
C\left(\Psi_0+\|\nabla^2u\|_{L^2(B_1)}+\|\nabla u\|_{L^4(B_1)}\right)\\
&\leq&
C\left(\|\nabla^2u\|_{L^2(B_1)}+\|\nabla u\|_{L^4(B_1)}\right).
\end{eqnarray*}

We start with $p=\frac{16}{13}$. The above argument implies that
\begin{equation*}
	\norm{u}_{W^{4,\frac{16}{13}}(B_{7/8})}\leq C (\norm{\nabla^2 u}_{L^2(B_1)} +\norm{\nabla u}_{L^4}(B_1)).
\end{equation*}
The Sobolev embedding theorem implies
\begin{equation*}
	\norm{\nabla^3 u}_{L^{\frac{16}{9}}(B_{7/8})} + \norm{\nabla^2 u}_{L^{\frac{16}{5}}(B_{7/8})} +\norm{\nabla u}_{L^{16}(B_{7/8})}\leq C (\norm{\nabla^2 u}_{L^2(B_1)} +\norm{\nabla u}_{L^4}(B_1)).
\end{equation*}
With this, we can bound the $L^{\frac{8}{5}}$ norm of the right hand side of the Euler-Lagrange equation. The interior $L^p$ estimate then shows $u$ is bounded in $W^{4,\frac{8}{5}}$ in $B_{3/4}$. The lemma is then proved by bootstrapping method.
\end{proof}

\end{document}